\documentclass{amsart}
\usepackage{enumerate}

\numberwithin{equation}{section}

\newtheorem{theorem}{Theorem}[section]
\newtheorem{corollary}[theorem]{Corollary}
\newtheorem{definition}[theorem]{Definition}
\newtheorem{lemma}[theorem]{Lemma}
\newtheorem{proposition}[theorem]{Proposition}
\newtheorem{problem}[theorem]{Problem}

\newcommand{\M}{{\mathcal M}}

\numberwithin{equation}{section}

\begin{document}

\title[\emph{}]{$\Phi$-moment inequalities for independent and freely independent random variables}

\author[]{Yong Jiao  }
\address{School of Mathematics and Statistics, Central South University, Changsha 410075, China}
\email{jiaoyong@csu.edu.cn}
\curraddr{School of Mathematics and Statistics, University of NSW, Sydney,  2052, Australia}
\email{yong.jiao@unsw.edu.au}

\author[]{Fedor Sukochev}
\address{School of Mathematics and Statistics, University of NSW, Sydney,  2052, Australia}
\email{f.sukochev@unsw.edu.au}

\author[]{Guangheng Xie  }
\address{School of Mathematics and Statistics, Central South University, Changsha 410075, China}
\email{xiegh@csu.edu.cn}

\author[]{Dmitriy  Zanin}
\address{School of Mathematics and Statistics, University of NSW, Sydney,  2052, Australia}
\email{d.zanin@unsw.edu.au}



\thanks{ Yong Jiao is supported by NSFC(11471337), Hunan Provincial Natural Science Foundation(14JJ1004) and The International Postdoctoral Exchange Fellowship Program. }

\begin{abstract}
This paper is devoted to the study of $\Phi$-moments of sums of independent/freely independent  random variables. More precisely, let $(f_k)_{k=1}^n$ be a sequence of positive (symmetrically distributed) independent random variables and let $\Phi$ be an Orlicz function with $\Delta_2$-condition. We provide an equivalent expression for the quantity $\mathbb{E}(\Phi(\sum_{k=1}^n f_k))$ in term of the sum of disjoint copies of the sequence $(f_k)_{k=1}^n.$ We also prove an analogous result in the setting of free probability. Furthermore, we provide an equivalent characterization of $\tau(\Phi(\sup^+_{1\leq k\leq n}x_k))$ for positive freely independent random variables and also present some new results on free Johnson-Schechtman inequalities in the quasi-Banach symmetric operator space.
\end{abstract}

\maketitle

\section{Introduction}

The main theme of this article is twofold: it concerns  $\Phi$-moment estimates of independent random variables and of freely independent self-adjoint operators 
affiliated with a finite von Neumann algebra.

In order to explain the classical probability theory roots of our study, we need to recall two outstanding results published simultaneously in 1970, due to Kruglov \cite{Kruglov} and Rosenthal \cite{R}, respectively. The results in  \cite{Kruglov} were concerned with infinitely divisible distributions occurring in the analysis of the classical Levy-Khintchine formula. Let $f$ be a random variable on $(0,1)$, and let $\pi(f)$ denote the random variable $\sum_{k=1}^Nf_k,$ where $f_k,$ $k\geq1,$ are independent copies of $f$ and $N$ is a Poisson random variable independent from the sequence $(f_k)$. In \cite{Kruglov}, the following $\Phi$-moment theorem was proved. 

\begin{theorem}[Kruglov Theorem]
\label{Kruglov}
Suppose that $\Phi$ is a positive continuous function on $\mathbb{R}$ with $\Phi(0)=0$ and suppose that it satisfies one of the following conditions.
\begin{enumerate}[{\rm (i)}]
\item\label{kra} $\Phi(t+s)\leq B\Phi(t)\Phi(s)$ for every $s,t\in\mathbb{R}$ and some constant $B>0$.
\item\label{krb} $\Phi(t+s)\leq B(\Phi(t)+\Phi(s))$ for every $s,t\in\mathbb{R}$ and some constant $B>0$.
\end{enumerate}
For an arbitrary random variable $f$ the conditions $\mathbb{E}(\Phi(f))<\infty$ and $\mathbb{E}(\Phi(\pi(f)))<\infty$ are equivalent.
\end{theorem}

The Rosenthal theorem \cite{R}  concerns only the special case $\Phi(t)=|t|^p$, $p\ge 1$ and was established while studying the $L_p$-norm of a sum of independent
functions.

\begin{theorem}[Rosenthal Theorem]
\label{theor: Rosenthal}


If $p> 2,$ then there exists a constant $K_p>0$ such that for
an arbitrary sequence $\{f_k\}_{k=1}^\infty\subset L_p$ of
independent functions satisfying $\int_0^1f_k(t)\,dt=0$
$(k=1,2,\dots)$ and for every $n\ge 1$ the following estimates
hold:
\begin{equation*}
\frac12\max\left\{\left(\sum_{k=1}^n\|f_k\|_p^p\right)^{1/p},
\left(\sum_{k=1}^n\|f_k\|_2^2\right)^{1/2}\right\} \le
\Big\|\sum_{k=1}^n f_k\Big\|_p\le
\end{equation*}
\begin{equation}
\le K_p\max\left\{\left(\sum_{k=1}^n\|f_k\|_p^p\right)^{1/p},
\left(\sum_{k=1}^n\|f_k\|_2^2\right)^{1/2}\right\}.\label{eq2imp}
\end{equation}
\end{theorem}


These two seemingly disconnected results, are in fact deeply connected. To explain better this connection, we need to refer to works of several other mathematicians. 
For convenience, let us introduce the notation $$\bigoplus _{k=1}^nf_k:=\sum_{k=1}^nf_k(\cdot-k+1)\chi_{[k-1,k)}$$
for disjoint sum of random variables $(f_k)$ on $(0,1)$, which is a Lebesgue measurable function on $(0,\infty)$. Then Theorem \ref{theor: Rosenthal} can be restated as
\begin{equation}\label{rosenthal}
\Big\|\sum_{k=1}^nf_k\Big\|_p\approx_{K_p} \Big \|\bigoplus _{k=1}^nf_k\Big\|_{L_p\cap L_2},\quad 2<p<\infty.
\end{equation}
Here and in what follows, $X\approx_C Y$ means that $X\leq CY$ and $Y\leq CX.$ This form of Theorem \ref{theor: Rosenthal} was extended by Carothers and Dilworth \cite{CD1,CD2}  to the case of Lorentz spaces $L_{p,q}$ 
and in 1989, in the setting of symmetric function spaces, Johnson and Schechtman \cite{JS} established a far reaching generalisation of Rosenthal's result for Banach and quasi-Banach symmetric function spaces  (see next section for precise definitions of these notions and subsequent terms and symbols). Let $E$ be a symmetric space on $[0,1]$ and  let the set $Z_E^p$ consists of all measurable functions on $(0,\infty)$ for which 
$$\|f\|_{Z_E^p}=\|\mu(f)\chi_{[0,1]}\|_E+\|\mu(f)\chi_{(1,\infty)}\|_p<\infty,$$
where $\mu(f)$ is a decreasing rearrangement of $|f|$.
It is established in \cite{JS}  that 
\begin{equation}\label{JS}
\Big\|\sum_{k=1}^nf_k\Big\|_E\approx_{C_E} \Big \|\bigoplus _{k=1}^nf_k\Big\|_{Z_E^2},\mbox{ respectively, }\Big\|\sum_{k=1}^nf_k\Big\|_E\approx_{C_E} \Big \|\bigoplus _{k=1}^nf_k\Big\|_{Z_E^1}\quad\Big)
\end{equation} for any sequence $(f_k)$ of independent mean zero (respectively, positive) random variables whenever $L_p\subset E$ for some $p<\infty.$ 
The connections between Johnson-Schechtman form \eqref{JS} of Rosenthal Theorem \ref{theor: Rosenthal} and Kruglov Theorem \ref{Kruglov}  was firstly noted by Braverman \cite{B}. However, in the setting of Banach and quasi-Banach symmetric function spaces, for detailed discussion of these connections we refer the reader to \cite{Astashkin2005, AS1, AS3, AS4, ASW, pacific}. The main tool used in the latter papers, the so-called Kruglov operator $K_{\rm class}$, allowed to substantially extend Johnson-Schechtman inequalities  \eqref{JS}. In particular, it follows from \cite[Theorems 3.5 and 6.1]{Astashkin2005} and \cite[Theorem 1]{AS4} that \eqref{JS} holds if and only if the operator $K_{\rm class}$ is bounded on $E.$ The latter condition is far less restrictive than the assumption that $X \supset L_p$ for some $p<\infty$ (see \cite{Astashkin2005}). To see (finally!) the connection between Theorems  \ref{theor: Rosenthal} and \ref{Kruglov}, it remains to observe that the implication $f\in E\Longrightarrow \pi(f)\in E$ holds if and only if the operator $K_{\rm class}$ acts boundedly on (symmetric Banach function space) $E$ \cite[p.1990]{AS4}.

Since the Kruglov operator $K_{\rm class}$ is bounded on $L_p$ for $1\leq p<\infty,$ (see e.g. precise estimates in \cite[Corollary 4]{AS4}), we may view  \eqref{rosenthal}  as 
an assertion implied by the second part of Theorem \ref{Kruglov} (obviously, $\Phi(t)=|t|^p$, $p\ge 1$ satisfies the condition stated in that part).

We are now in a position to state the first main question studied in this article. 

{\sl For which Orlicz functions $\Phi$, the $\Phi$-moment versions of inequalities \eqref{JS} remain valid?} 

This question actually returns to the very original setting of Kruglov Theorem. It should be also emphasized that the answer to this question is dramatically different from its counterpart for the Banach space setting studied in papers cited above. In the present paper, we answer this question in full generality. (Throughout this article, $C_\Phi$ always denotes a constant depending only on $\Phi,$ which may be different in different places).


\begin{theorem}\label{main results 1}
Suppose that $\Phi$ is an Orlicz function satisfying $\Delta_2$-condition. Let $\{f_k\}_{k=1}^n \subset L_{\Phi}[0,1],$ $n\in\mathbb{N},$ be a sequence of independent random variables.
\begin{enumerate}[{\rm (i)}]
\item If $f_k,$ $1\leq k\leq n,$ are positive, then
\begin{equation}\label{mainca}
\mathbb{E}\Big(\Phi\big(\sum_{k=1}^n {f_k}\big)\Big)\approx_{C_\Phi}\mathbb{E}\Big(\Phi \big(\mu(f)\chi_{(0,1)} \big)\Big)+\Phi\Big(\|f\|_1\Big),\quad f=\bigoplus_{k=1}^n f_k.
\end{equation}
\item If $f_k,$ $1\leq k\leq n,$ are symmetrically distributed, then
\begin{equation}\label{maincb}
\mathbb{E}\Big(\Phi\big(\sum_{k=1}^n {f_k}\big)\Big)\approx_{C_\Phi}\mathbb{E}\Big(\Phi \big(\mu(f)\chi_{(0,1)} \big)\Big)+\Phi\Big(\|f\|_{L_1+L_2}\Big),\quad f=\bigoplus_{k=1}^n f_k.
\end{equation}
\end{enumerate}
\end{theorem}

We also prove the following maximal inequalities, which extend Lemma 1 in \cite{MS}. 
\begin{theorem} \label{CM}
Suppose that $\Phi$ is an Orlicz function satisfying $\Delta_2$-condition. If $\{f_k\}_{k=1}^n \subset L_{\Phi}[0,1],$ $n\in\mathbb{N}$ is a sequence of positive independent random variables, then
\begin{equation}
\mathbb{E}\Big(\Phi\big(\max_{1\leq k\leq n}f_k\big)\Big)\approx_{C_\Phi}\mathbb{E}\Big(\Phi \big(\mu(f)\chi_{(0,1)} \big)\Big),\quad f=\bigoplus_{k=1}^n f_k.
\end{equation}
\end{theorem}

Now we turn to the second main theme of this paper, which concerns analogues of the classical results discussed above in the setting of free probability. Recently, the operator approach of \cite{AS1} was extended into the realm of (noncommutative) free probability theory in \cite{SZ}. By using a free Kruglov operator $K_{\rm free},$ a version of Johnson-Schechtman inequalities in the setting of free probability theory was obtained. We briefly recall the main results from \cite{JS}. Let $\M$ be a finite von Neumann algebra equipped with a faithful normal tracial state $\tau$. Let $E(\M, \tau)$ be a symmetric Banach operator space equipped with a Fatou norm. Then\footnote{In the right hand side, we consider the norms in the symmetric operator space $Z_E^2$ associated with the algebra $\mathcal{M}\otimes l_{\infty}$ and the trace $\tau\otimes\#,$ where $\#$ is the counting measure on $\mathbb{Z}_+.$ The notation $\bigoplus_{k=1}^nx_k$ is a shorthand for disjoint sum of the operators $x_k,$ $1\leq k\leq n,$ (e.g.  $\bigoplus_{k=1}^nx_k$= $\sum_{k=1}^nx_k\otimes e_k$, where $e_k$ is the standard basis of $\ell_\infty$).}
\begin{equation}\label{SZ}
\Big\|\sum_{k=1}^nx_k\Big\|_E\approx\Big\|\bigoplus_{k=1}^nx_k\Big\|_{Z_E^2},\mbox{ respectively, }\Big\|\sum_{k=1}^nx_k\Big\|_E\approx\Big\|\bigoplus_{k=1}^nx_k\Big\|_{Z_E^1}
\end{equation}
for every sequence $(x_k)$ of freely independent symmetrically distributed (respectively, positive) random variables from $E(\M,\tau);$ see section 2 for the notations.  In the special case $E=L_p$ for $1\leq p\leq \infty$, this result was proved by Junge, Parcet and Xu (see Theorem A in \cite{JPX}). For $p=\infty$, it belongs to Voiculescu \cite{V1}. The second aim of this article is to answer the question

{\sl For which Orlicz functions $\Phi$, the $\Phi$-moment versions of inequalities \eqref{SZ} remain valid?} 

and to present a noncommutative $\Phi$-moment version of Johnson-Schechtman inequalities, namely, a free version of Theorem \ref{main results 1}.

This question is fully answered in the following theorem. The Banach space $L_1+L_2$ is the standard sum of Banach spaces $L_1$ and $L_2$, which is given by the set $\{h=f+g|\ f\in L_1,\ g\in L_2\}$ equipped with the norm $\|h\|_{L_1+L_2}=\inf \|f\|_{L_1}+\|g\|_{L_2}$, where the infimum is taken over all possible decompositions $h=f+g$, $f\in L_1$, $g\in L_2$.

\begin{theorem} \label{main results 2} Let $(\mathcal{M},\tau)$ be a noncommutative probability space. Suppose that $\Phi$ is an Orlicz function satisfying $\Delta_2$-condition. Let $\{x_k\}_{k=1}^n,$ $n\in\mathbb{N},$ be a sequence of freely independent random variables.
\begin{enumerate}[{\rm (i)}]
\item\label{mainfa} If $x_k,$ $1\leq k\leq n,$ are positive, then
\begin{equation}
\tau(\Phi(\sum_{k=1}^nx_k))\approx_{C_\Phi}\mathbb{E}(\Phi(\mu(X)\chi_{(0,1)}))+\Phi(\|X\|_1),\quad X=\bigoplus_{k=1}^n x_k.
\end{equation}
\item\label{mainfb} If $x_k,$ $1\leq k\leq n,$ are symmetrically distributed, then
\begin{equation}
\tau(\Phi(\sum_{k=1}^nx_k))\approx_{C_\Phi}\mathbb{E}(\Phi(\mu(X)\chi_{(0,1)}))+\Phi(\|X\|_{L_1+L_2}),\quad X=\bigoplus_{k=1}^n x_k.
\end{equation}
\end{enumerate}
\end{theorem}

For simplicity, the following statement of a noncommutative counterpart of Theorem \ref{CM} does not refer to the notion of noncommutative maximal operator, which was introduced by Pisier in \cite{P} (see also Junge's paper \cite{J}). For recent studies of $\Phi$-moments of noncommutative maximal operator we refer to \cite{BCO} (see also \cite{BC, BCLJ}). 

\begin{theorem}\label{Mright} Let $\Phi$ be an Orlicz function satisfying $\Delta_2$-condition. Let $(x_k)_{k=1}^n\subset S(\mathcal{M},\tau)$ be a sequence of positive freely independent random variables. We have\footnote{In the notations of Pisier \cite{P} and Junge \cite{J}, the left hand side is written as $\tau(\Phi(\sup^+_{1\leq k\leq n}x_k)).$ We chose not to use this notation because the object $\sup^+_{1\leq k\leq n}x_k$ does not exist {\it per se}.}
$$\inf\big\{\tau(\Phi(a)):\ a\geq x_k,\ 1\leq k\leq n\big\}\approx_{C_{\Phi}} \mathbb{E}(\Phi(\mu(X)\chi_{(0,1)})),\quad X=\bigoplus_{k=1}^nx_k.$$
\end{theorem}

Finally, we explain why Theorems \ref{main results 1} and \ref{main results 2} are sharp in the sense that $\Delta_2$-condition is necessary. The notations $L_{\Phi}[0,1]$ and 
$L_\Phi(\mathcal{M},\tau)$ stand for Orlicz spaces associated with function $\Phi$ and $\mathbb P$ stands for Lebesgue measure.

\begin{theorem}\label{sharp}
Suppose that $\Phi$ is an Orlicz function. Suppose that either
\begin{enumerate}[{\rm (i)}]
\item for every sequence $\{f_k\}_{k=1}^n$ of independent random variables from $L_{\Phi}[0,1]$ satisfying the condition $\sum_{k=1}^n\mathbb P(\text{supp}(f_k))\leq1,$ we have
\begin{equation}\label{sharpa}
 \mathbb{E}\Big(\Phi \big(\sum_{k=1}^n f_k\big)\Big) \leq C_\Phi\mathbb E\Big(\Phi \big(\bigoplus_{k=1}^n f_k\big)\Big);
\end{equation}
or
\item for every sequence $\{x_k\}_{k=1}^n$ of freely independent random variables from $L_\Phi(\mathcal{M},\tau)$ satisfying the condition $\sum_{k=1}^n\tau(\text{supp}(x_k))\leq1,$ we have
\begin{equation}\label{sharpb}
\tau\Big(\Phi\big(|\sum_{k=1}^n x_k|\big)\Big)\leq C_\Phi\tau\Big(\Phi\big(\bigoplus_{k=1}^n x_k\big)\Big).
\end{equation}
\end{enumerate}
Then $\Phi$ satisfies $\Delta_2$-condition.
\end{theorem}

The theorem above allows us to make some interesting comparisons between modular inequalities as stated above and Orlicz norm inequalities from \cite{Astashkin2005,AS1,AS3,AS4,ASW,pacific,B}. It is well known that every Orlicz space 
($L_{\Phi}[0,1]$, $\|\cdot\|_\Phi$) with the function $\Phi$ satisfying the $\Delta_2$-condition contains  $L_p[0,1]$ for some $p<\infty$ and therefore \eqref{JS}  holds for such space. Hence, the operator $K_{\rm class}$ is bounded on $L_{\Phi}[0,1]$ \cite{JS} and so the conditions $\mathbb{E}(\Phi(f))<\infty$ and 
$\mathbb{E}(\Phi(\pi(f)))<\infty$ are equivalent for every  $f\in L_{\Phi}[0,1]$ by Theorem \ref{Kruglov} (see \cite{B} and \cite{Astashkin2005}). 
Our new results in \eqref{mainca} and \eqref{maincb} complement this line of thought. However, there exists a significant difference between symmetric norm estimates \eqref{JS} and 
$\Phi$-moment estimates studied in the present paper. Indeed, suppose that $\Phi$ does not satisfy the $\Delta_2$-condition but still satisfies the assumption of Theorem \ref{Kruglov} \eqref{kra}. Then the inequality
$$\|\sum_{k=1}^n f_k\|_\Phi \leq C_\Phi \|\bigoplus_{k=1}^n f_k\|_\Phi$$
holds for any sequence $\{f_k\}_{k=1}^n$ of independent random variables from $L_{\Phi}[0,1]$ satisfying the condition $\sum_{k=1}^n\mathbb P(\text{supp}(f_k))\leq1$, whereas the similar inequality \eqref{sharpa} for $\Phi$-moments fails due to Theorem \ref{sharp}(a)! This unexpected result indicates a substantial difference between Johnson-Schechtman inequalities for symmetric norms and for $\Phi$-moments.


We complete this introduction by referring the reader to \cite{K} for a different approach to bounds of $\Phi$-moments of sum of independent random variables (see also Theorem 1.4.11 in \cite{LG}), and to \cite{BC,BCLJ,D,DR} for some progress on noncommutative $\Phi$-moment martingale inequalities.

This paper is further divided into seven sections. In Section 2, we present some classical and noncommutative notations  and preliminaries. Section 3 is devoted to proving Theorem \ref{main results 1}, Theorem \ref{CM} and Theorem \ref{sharp} \eqref{sharpa}; Theorem \ref{main results 2}  and Theorem \ref{sharp} \eqref{sharpb} are proved in Section 4. In Section 5 we discuss the noncommutative maximal inequalities and prove Theorem \ref{Mright}. In Section 6, we mainly prove some new results on free Johnson-Schechtman inequalities in the setting of quasi-Banach symmetric operator spaces. Finally, we state an open problem in Section 7.

\section{Preliminaries}

\subsection{Orlicz functions and classical Kruglov operator}

For a Lebesgue measurable, a.e. finite function $f$ on $(0,1)$ (or $(0,\infty)$) we define its $distribution$ $function$ by
$$\lambda(s,f):=\mathbb P\{t:f(t)>s\},\quad s\in\mathbb{R},$$
where $\mathbb P$ stands for Lebesgue measure. Let $S(0,1)$ (respectively, $S(0,\infty)$) denote the space of all Lebesgue measurable functions on $(0,1)$.

Two measurable functions $f$ and $g$ are called {\it equimeasurable} if both $\lambda(f_+)=\lambda(g_+)$ and $\lambda(f_)=\lambda(g_-)$ (if the functions live on $(0,1),$ then this is equivalent to $\lambda(f)=\lambda(g)$). Here, $f_+=f\vee 0$ and $f_-=-f\vee 0.$ In particular, for every measurable function $f,$ the function $|f|$ is equimeasurable with its {\it decreasing rearrangement} $\mu(f),$ defined by the formula
$$\mu(t,f):=\inf\{u\geq0:\lambda(u,|f|)<t\},\quad t>0.$$
If $f,g\geq0$, then $\mu(f)=\mu(g)$ if and only if $f$ and $g$ are equimeasurable. We recall that a function $f$ is said to be symmetrically distributed if $f$ and $-f$ are equimeasurable. Let $0\leq f, g \in L_1(0,1)$.

\begin{definition}
Let $E\subset S(0,1)$ (or $E\subset S(0,\infty)$) be a quasi-Banach space.
\begin{enumerate}[{\rm (1)}]
\item $E$ is said to be a quasi-Banach function space if, from $f\in E,$ $g\in S(0,1)$ (or $g\in S(0,\infty)$) and $|g|\leq|f|$, it follows that $g\in E $ and $\|g\|_E \leq\|f\|_E.$
\item A quasi-Banach function space $E$ is said to be symmetric if, for every $f\in E $ and any measurable function $g$, the assumption $\mu(g)=\mu(f)$ implies that $g\in E$ and $\|g\|_E=\|f\|_E.$
\end{enumerate}
\end{definition}
Without loss of generality, in what follows we assume that $\|\chi_{(0,1)}\|_E=1$, where $\chi_A$ denotes the indicator function of a Lebesgue measurable set $A.$

The following useful construction may be found in  \cite{johnson, JS, Astashkin2005}. If $E$ is a quasi-Banach symmetric function space on $(0,1)$ and $0<p\leq \infty$, 
then the space $Z_E^p$ consists of all $f\in S(0,\infty)$ such that
$$\|f\|_{Z_E^p}:=\|\mu(f)\chi_{(0,1)}\|_E+\|\mu(f)\chi_{(1,\infty)}\|_p<\infty.$$
It is not difficult to check that the functional $\|\cdot\|_{Z_E^p}$ is a quasi-norm on $Z_E^p.$

\begin{definition} Let $f,g\in L_1(0,1)$. We write $g\prec\prec f$ if
$$\int_0^t\mu(s,g)ds\leq\int_0^t\mu(s,f)ds,\quad t>0.$$
If, in addition, $f,g\geq0$ and $\|f\|_1=\|g\|_1,$ then we write $g\prec f.$
\end{definition}

The following assertion is Lemma 13 in \cite{pacific}.

\begin{lemma}\label{pacific lemma} Let $\{f_k\}_{k=1}^n$ and $\{g_k\}_{k=1}^n$, $n\in \mathbb{N}$, be sequences of positive and independent functions from $L_1(0,1)$. If $g_k\prec f_k$, for each $k$, $1\leq k\leq n$, then
$$\sum_{k=1}^n g_k\prec \sum_{k=1}^n f_k.$$
\end{lemma}
In Section \ref{jsqb} below, we prove a free version of Lemma \ref{pacific lemma}.

Let $f_k$, $k\geq1$, be elements from $S(0,1)$ and let $g_k\in S(0,\infty)$, $k\geq1$, be their disjoint copies; that is, $f_k$ is equimeasurable with $g_k$ for all $k\geq1$, and $g_lg_m=0$ if $l\neq m$. For example, we can set $g_k(t)=f_k(t-k+1)\chi_{[k-1,k)}(t)\ (t>0)$. For the function $\sum_{k\geq0}g_k$, which is frequently called the $disjoint$ $sum$ of $f_k,k\geq1$, we  use the suggestive notation $\bigoplus _{k\geq1}f_k$. It is important to observe that the distribution function of a disjoint sum $\bigoplus _{k\geq1}f_k$
does not depend on the particular choice of elements $g_k,\ k\geq1$. Obviously, the disjoint sum has the following property:
$$\mathbb P\big\{s>0:\bigoplus _{k=1}^n f_k(s)>\lambda\big\}=\sum_{k=1}^n\mathbb P\{s\in(0,1):f_k(s)>\lambda\}.$$
In the special case when $\sum_{k=1}^n \mathbb P\{\text{supp}(f_k)\}\leq1,n\in \mathbb{N}$, it is convenient to view the disjoint sum $\bigoplus _{k\geq1}f_k$ as a measurable function on $(0,1).$ 

The following result is well known; see for instance Lemma 3 in \cite{JS}. It plays an important role in the proof of the classical Johnson-Schechtman theorem, but unfortunately, we do not have its free version, which we discuss (see Problem 7.1) at the end of this article.  We recall that the dilation operator $\sigma_s:S(0,1)\rightarrow S(0,1)$, $s\in (0,1)$ is given by 
$(\sigma_sx)(t)=x(\frac{t}{s})$ if $t\in (0,s)$; otherwise $(\sigma_sx)(t)=0$. 

\begin{lemma}\label{wbjohn lemma}
Let $\{f_k\}_{k=1}^n$ be a sequence of positive independent random variables on $[0,1]$ such that $\sum_{k=1}^n\mathbb{P}({\rm supp}(f_k))\leq 1.$ We  have
$$\mu(\bigoplus_{k=1}^nf_k)\leq\sigma_2\mu(\max_{1\leq k\leq n}f_k)\mbox{ and, therefore, }\mu(\bigoplus_{k=1}^nf_k)\leq\sigma_2\mu(\sum_{k=1}^nf_k).$$
\end{lemma}

Let $\Phi:\mathbb{R}\to\mathbb{R}_+$ be an Orlicz function, i.e. an even convex function such that $\Phi(0)=0.$ 
In the following, we will use the notation ${\mathbb E}\Phi= \int_0^1
\Phi(s)\,ds.$ By $L_{\Phi}$ we denote the class of all measurable functions $f$ on $[0,\infty)$ such that the norm
$$\|f\|_{L_{\Phi}}:=\inf\Big\{\lambda>0:\int_0^{\infty}\Phi(\frac{f(t)}{\lambda})dt\leq1\Big\}$$
is finite. It is well known that $L_\Phi$ is a symmetric function space \cite{KPS}. An Orlicz function $\Phi$ satisfies the $\Delta _2$-condition 
if there is a constant $C$ such that $\Phi(2t)\leq C\Phi(t)$ for all $t>0.$ In this case, for every $a>0$ there is a constant $C_a>0$ such that $\Phi(at) \leq C_a\Phi(t)$ for all $t>0.$ Equivalently, an Orlicz function $\Phi$ satisfies the $\Delta_2$-condition if and only if
\begin{equation}\label{ineq13}
\Phi(u+v)\le \gamma(\Phi(u)+\Phi(v))
\end{equation}
for some constant $\gamma>0$ and all $u,v\geq 0$ (see for e.g. \cite[Formula (7.9)]{Burk}). In particular, any such functions satisfies the condition of Theorem \ref{Kruglov} \eqref{krb}.

It is well known (see e.g. \cite[Theorem 11]{DSZ} and references therein) that
$$g\prec\prec f\Longrightarrow {\mathbb E}\Phi(g)\leq {\mathbb E}\Phi(f).$$

Before introducing the definition of the Kruglov operator, we consider the probability product space
$$(\Omega,\mathbb{P}):=\prod_{k=0}^{\infty}((0,1),\mathbb{P}_k),$$
($\mathbb P_k$ is the Lebesgue measure on $(0,1),k\geq0$). Observe that in an arbitrary symmetric space the norms of any two elements with identical distribution coincide. Hence, using a bijective measure-preserving transformation between measure space $(\Omega,\mathbb{P})$ and $((0,1),\mathbb P)$, we  identify an arbitrary measurable function $f(\omega)=f(\omega_0,\omega_1,\cdots,\omega_n\cdots)$ on $(\Omega,\mathbb{P})$ with the corresponding element from $S(0,1)$. A particular form of the measure-preserving transformation used in such identification does not play any role and we completely suppress it from the notations. Thus, we  view the set $\Omega$ as $(0,1)$ and any measurable function on $(\Omega,\mathbb{P})$ as a function from $S(0,1)$.


Now, we are ready to explain the notion of the Kruglov operator introduced in \cite{AS1}. Let $\{A_n\}_{n=0}^{\infty}$ be a fixed sequence of mutually disjoint measurable subsets of $(0,1)$ such that $\mathbb P(A_n)=\frac1{e\cdot n!}.$ Define the operator $K_{\rm class}:S(0,1)\rightarrow S(0,1)$
by setting
\begin{equation}\label{definitionofKclass}
 K_{\rm class}f(\omega):=\sum_{n=1}^{\infty}\sum_{k=1}^n f(\omega _k)\chi_{A_n}(\omega_0).
\end{equation}

In this paper, Kruglov operator $K_{\rm class}$ is an important tool to compare sums of independent functions with sums of their disjoint copies. The following assertion appeared yet in \cite{Astashkin2005}, but the very first proof was given only in \cite{ASW}. 

\begin{theorem}\label{DFclass}
 If $f_k\in S(0,1)$, $1\leq k \leq n$, are disjointly supported functions, then $K_{\rm class}f_k$, $1\leq k \leq n$, are independent random variables.
\end{theorem}

\subsection{Noncommutative probability spaces and freely independent random variables}

In this subsection we introduce some basic definitions and well-known results concerning  noncommutative probability space and free independence. We use standard notation for operator algebras  as may be found in the books \cite{KR,T}.  Let $\M$ be  a finite von Neumann algebra  equipped with a  faithful normal trace $\tau$. If $\tau(1)=1$, we call $(\M,\tau)$ be a noncommutative probability space.  
Let $S(\M,\tau)$ denote the $*$-algebra of all $\tau$-measurable operators with respect to $(\M,\tau)$ \cite{FK}. The topology of $S(\M,\tau)$ is determined by the convergence in measure (see e.g. \cite{FK}). Let $S^h(\M,\tau)$ denote the set of all self-adjoint elements in $S(\M,\tau),$ which are called (noncommutative) random variables.  For $x\in S^h(\M,\tau)$, the distribution function $\lambda(x)$ is defined by the formula
$$\lambda(s,x)=\tau\big(e_{(s,\infty)}(x)\big),\quad -\infty<s<\infty,$$
where $e_{(s,\infty)}(x)$ is the spectral projection of $x$ associated with the interval $(s,\infty)$. If $\mathcal{M}=L_{\infty}(0,1),$ 
then the latter definition coincides with the distribution function defined in the preceding subsection. 
The generalized singular value function $\mu(x): t\longrightarrow \mu(t,x)$, $x\in S(\M,\tau)$ is defined by 
$$\mu(t,x)=\inf\big\{s>0, \lambda(s,|x|)\leq t \big\},\quad 0\leq t\leq1.$$
Random variables $x, y\in S^h(\M,\tau)$ are said to be equimeasurable if $\lambda(x)=\lambda(y)$. For such operators, we also have $\mu(x)=\mu(y).$ We refer to \cite{FK} for further information on the generalised singular value function.

We say that random variables $x_k$ converge to $x$ in distribution if $\lambda(x_k)\to\lambda(x)$ almost everywhere. In particular, $\mu(x_k)\to\mu(x)$ almost everywhere. For the uniformly bounded sequence of random variables convergence in distribution is equivalent to the convergence of moments. Indeed, convergence of moments implies that
$$\tau(e^{itx_k})\to\tau(e^{itx})$$
for every $t\in\mathbb{R},$ that is, characteristic functions of the random variables $x_k$ converge to that of $x.$ Using result on convergence of characteristic functions proved in section II.13 in \cite{GnedKol}, we conclude the convergence in distribution.

Let $\Phi$ be an Orlicz function on $\mathbb R$.
For any $x\in S(\M,\tau)$, we have by \cite[Corollary 2.8]{FK}
\begin{equation}\label{traceequality}
 \tau(\Phi(|x|))=\int_0^\infty\Phi(\mu(t,x))dt.
\end{equation}

If Orlicz function $\Phi$ satisfies $\Delta_2$-condition, then combining \eqref{ineq13} with Theorem 4.4(iii) in \cite{FK} we obtain that for all random variables $x,y\in S(\M,\tau),$ we have
\begin{equation}\label{noncommineq13}
 \tau(\Phi(x+y))\leq C_\Phi(\tau(\Phi(x))+\tau(\Phi(y)).
\end{equation}

It follows from Jensen inequality  that
\begin{equation}\label{Jensen}
 \tau(\Phi(x))\geq\Phi(\tau(x)).
\end{equation}

We frequently use these inequalities in the sequel sometimes even without additional references.

We now introduce the free independence and free Kruglov operator.

Let $(\M,\tau)$ be a noncommutative probability space. The von Neumann subalgebras $\M_i, i\in I,$ of $\M$ are freely independent (with respect to $\tau$) if $\tau(x_1\cdots x_n)=0$ whenever $x_j\in \M_{i_j}$, $i_1\neq i_2\neq\cdots\neq i_n$ and $\tau(x_j)=0$ for $1\leq j\leq n$ and every $n\in \mathbb N$. A family of random variables $\{x_1,\cdots,x_n\}$ is said to be freely independent if the von Neumann subalgebras generated by $x_j$ are freely independent.

We frequently use the following fact: if $\{x_j\}_{j=1}^n$ is a sequence of freely independent random variables and if $y_j$ belongs to the von Neumann algebra generated by $x_j,$ then $\{y_j\}_{j=1}^n$ is also a sequence of freely independent random variables.

The following definition is taken from \cite{VDN} (see also \cite{NS1}).

\begin{definition} A von Neumann algebra $\M$ is called a free product of von Neumann subalgebras $\M_i, i\in I$ if for every von Neumann algebra $\mathcal{N}$ and for every set of unital $*$-homomorphisms $\pi_i:\M_i\rightarrow \mathcal{N}$ there exists a unique unital $*$-homomorphism $\pi:\M\rightarrow \mathcal{N}$ such that $\pi_i=\pi|_{\M_i}$.  The free product of von Neumann subalgebras $\M_i, i\in I$ is denoted by $\star_{i\in I}\M_i.$
\end{definition}

Let $\mathbb{F}_\infty$ be a free group with countably many generators. We denote the group von Neumann algebra associated with $\mathbb{F}_\infty$ equipped with canonical trace by $(L_\infty(\mathbb{F}_\infty),\tau).$ The algebra is known to be a finite factor and satisfies $L_\infty(\mathbb{F}_\infty)\simeq L_\infty(\mathbb{F}_\infty)\star L_\infty(\mathbb{F}_\infty).$ 

\begin{lemma}\label{FS}
Let $x_1,\cdots,x_n\in  S^h(\M,\tau) $ be freely independent random variables.  The distribution of $\sum_{k=1}^nx_k$ is uniquely determined by that of $x_k, 1\leq k\leq n.$
\end{lemma}

We frequently use the following lemma, which is a generalization of Schmidt decomposition for compact operators on a Hilbert space. Its proof can be found in \cite[Theorem 3.5]{DDP}  and \cite[Lemma 4.1]{CS}.

\begin{lemma}\label{embedding}
 Let $(\mathcal{M},\tau )$ be a finite
non-atomic von Neumann algebra. If $x\in S(\M,\tau)$ then there exists  a positive
rearrangement-preserving algebra $\ast $-isomorphism
$J_{\left\vert{x}\right\vert }$ of $S(0,1 )$ into $S(\M,\tau)$ such that
\begin{equation*}
J_{\left\vert x\right\vert }(\mu (x))=|x|.
\end{equation*}
In particular, for $f\in S(0,1)$,
\begin{equation*}
\mu(f)=\mu (J_{\left\vert x\right\vert }(f)).
\end{equation*}
\end{lemma}

We now introduce the free Kruglov operator defined in \cite{SZ} based on a very elegant construction given in \cite{NS2}. 
\begin {definition} Let $L_\infty(\mathbb{F}_\infty)=\M_0\star\mathcal{M}_1$ with $\M_i\simeq L_\infty(\mathbb{F}_\infty)$ for $i=0,1.$ Let $w\in \M_1$ be a semi-circular random variable, that is a random variable with distribution function supported on the interval $(-2,2)$ and with the density $\frac{1}{2\pi}(4-t^2)^{1/2}dt. $ Select a trace preserving isometric embedding of $S(0,1)$ into $S^h(\M_0,\tau).$ Define the Kruglov operator $K_{\rm free}:S(0,1)\rightarrow S(\mathbb{F}_\infty,\tau)$ as the restriction of the map $x\rightarrow wxw$ (from $S(\mathbb{F}_\infty,\tau)$ to $S(\mathbb{F}_\infty,\tau)$) to subalgebra $S(0,1)$ of the $S(\mathbb{F}_\infty,\tau).$
\end{definition}

We have the following important result; see  \cite[Corollary 1.8]{NS2} and \cite[Theorem 22]{SZ}.
\begin{theorem}\label{DF}
 If $x_k\in S(0,1)$, $1\leq k \leq n$, are disjointly supported functions, then $K_{\rm free}x_k$, $1\leq k \leq n$, are freely independent random variables.
\end{theorem}

We need the following technical estimate for the Kruglov operator. The following lemma coincides with \cite[Proposition 25]{SZ}.

\begin{lemma}\label{KB}
 For each positive $x\in S(0,1)$, we have 
$$\frac{2}{5}\sigma_{1/20}\mu(x)\leq \mu(K_{\rm free}x)\leq4\mu(x).$$
\end{lemma}

The following definition of free Poisson random variable is taken from p.35 in \cite{VDN}. The probability measure $m_u$ given by the formula
\begin{equation}\label{frpodef}
dm_u(t)=\frac1{2\pi t}\sqrt{4u-(t-1-u)^2}\chi_{((u^{1/2}-1)^2,(u^{1/2}+1)^2)}(t)dt
\end{equation}
defines a free Poisson random variable with a parameter $u.$ The following facts are easy to see from Lemma 24 in \cite{SZ} and from pp. 34-35 in \cite{VDN}, respectively.

\begin{lemma}\label{cumpois} For every $u\in(0,1),$ $K(\chi_{(0,u)})$ is a free Poisson random variable with parameter $u.$
\end{lemma}
\begin{proof} Lemma 24 in \cite{SZ} states that all free cumulants of $K(\chi_{(0,u)})$ are the same, namely, $\kappa_m(K(\chi_{(0,u)}))=u.$ By the definition of the $R-$transform (see Theorem 3.3.1 in \cite{VDN}) as applied to $K(\chi_{(0,u)}),$ we have $R_{\mu}(z)=\frac{u}{1-z}.$ This is exactly the $R-$transform of a free Poisson random variable with parameter $u$ as given on p.35 in \cite{VDN}. By Theorem 3.3.1 in \cite{VDN}, moments of $K(\chi_{(0,u)})$ coincide with that of a free Poisson random variable with parameter $u.$ So do their characteristic functions and, hence, their distributions.
\end{proof}

\begin{lemma}\label{frsum} If $\{\xi_k\}_{k=1}^n$ is a sequence of freely independent free Poisson random variables with parameters $\{u_k\}_{k=1}^n,$ then $\sum_{k=1}^n\xi_k$ is a free Poisson random variable with parameter $\sum_{k=1}^nu_k.$
\end{lemma}
\begin{proof} Lemma \ref{cumpois} states that all free cumulants of $K(\xi_k)$ are the same, namely, $\kappa_m(\xi_k)=u_k.$ By Proposition 12.3 in \cite{NS1}, we have that $\kappa_m(\sum_{k=1}^n\xi_k)=\sum_{k=1}^nu_k.$ Repeating the argument in Lemma \ref{cumpois}, we conclude the proof.
\end{proof}

\section{$\Phi$-moment inequalities for the classical independence}
In this section, we prove Theorem \ref{main results 1}, Theorem \ref{CM} and Theorem \ref{sharp} \eqref{sharpa}. We begin with the following $\Phi$-moment version of Johnson-Schechtman inequality in the commutative case.


%

\begin{theorem}\label{positive js rhs}
Suppose that $\Phi$ is an Orlicz function satisfying $\Delta_2$-condition.
If $\{f_k\}_{k=1}^n \subset L_{\Phi}[0,1],$ $n\in\mathbb{N}$ is a sequence of independent random variables, then
\begin{equation}
\mathbb{E}\Big(\Phi\big(\sum_{k=1}^n {f_k}\big)\Big)\leq C_\Phi\Big[\mathbb{E}\Big(\Phi \big(\mu(f)\chi_{(0,1)} \big)\Big)+\Phi\Big(\|f\|_1\Big)\Big],\quad f=\bigoplus_{k=1}^n f_k.
\end{equation}
\end{theorem}

The following lemmas are essential ingredients in the proof of Theorem \ref{positive js rhs}.

\begin{lemma}\label{kmaj} For every positive $f\in L_1(0,1),$ we have $f\prec K_{\rm class}f.$
\end{lemma}
\begin{proof} Using the identification between measure spaces $((0,1), \mathbb{P})$ and $(\Omega, \mathbb{P})$, we  rewrite the definition \eqref{definitionofKclass} as follows
$$K_{\rm class}f=\bigoplus_{n=1}^{\infty}\chi_{A_n}\otimes\Big(\sum_{k=1}^n 1^{\otimes (k-1)}\otimes f\otimes 1^{\infty}\Big).$$
Since $f\ge 0$ by the assumption, we easily infer from \cite[Lemma 3.3.7]{LSZ} that
$$f^{\oplus n}\prec\sum_{k=1}^n 1^{\otimes (k-1)}\otimes f\otimes 1^{\infty}.$$
By \cite[Lemma 2.3 ]{CDS}, we have 
$$\bigoplus_{n\geq 1}g_n\prec\prec\bigoplus_{n\geq1}h_n,\quad {\rm  if} \quad  g_n\prec\prec h_n, n\geq1.$$
Applying the above to sequences $\{\chi_{A_n}\otimes f^{\oplus n}\}_{n=1}^\infty$ and $\{\chi_{A_n}\otimes\Big(\sum_{k=1}^n 1^{\otimes (k-1)}\otimes f\otimes 1^{\infty}\Big)\}_{n=1}^\infty$, we arrive at
$$\bigoplus_{n=1}^{\infty}\chi_{A_n}\otimes f^{\oplus n}\prec K_{\rm class}f.$$
It remains to observe that the element standing on the left hand side above is equimeasurable with $f$. Indeed, fix an arbitrary scalar $\lambda>0$. We have 
$$
\mathbb{P}\{\chi_{A_n}\otimes f^{\oplus n} >\lambda\}=n\cdot \mathbb{P}(A_n)\mathbb{P}\{f >\lambda\},\quad n \ge 1.
$$
Recalling that $\mathbb {P}(A_n)=\frac{1}{e\cdot n!},\ n\ge 0$ we see that 
$$
\mathbb{P}\{\bigoplus_{n=1}^{\infty}\chi_{A_n}\otimes f^{\oplus n} >\lambda\}=\mathbb{P}\{f >\lambda\}.
$$
This completes the proof.
\end{proof}

Lemma below strengthens Theorem \ref{Kruglov} \eqref{krb}. We note in passing that if $\Delta_2$-condition is replaced by the condition from Theorem \ref{Kruglov} \eqref{kra}, then the following lemma  still holds; however we do not pursue this case here.

\begin{lemma}\label{phik} Suppose that $\Phi$ is an Orlicz function satisfying $\Delta_2$-condition. There exists a constant $C_{\Phi}$ such that for every positive random variable $f$ we have
$$\mathbb{E}(\Phi(K_{\rm class}f))\leq C_{\Phi}\mathbb{E}(\Phi(f)).$$
\end{lemma}
\begin{proof} By \eqref{ineq13} we have
$$\Phi(s+t)\leq B(\Phi(s)+\Phi(t)),\quad s,t>0.$$
By induction, we have
$$\Phi(\sum_{k=1}^ns_k)\leq B^{n-1}\sum_{k=1}^n\Phi(s_k),\quad 0\leq s_k,\ 1\leq k\leq n.$$
Therefore,
$$\mathbb{E}\Big(\Phi\big(\sum_{k=1}^n 1^{\otimes (k-1)}\otimes f\otimes 1^{\infty}\big)\Big)\leq B^{n-1}n\mathbb{E}(\Phi(f)).$$
By definition of the Kruglov operator, we have
$$\mathbb{E}\Big(\Phi\big(K_{\rm class}f\big)\Big)=\sum_{n=1}^{\infty}\frac1{e\cdot n!}\mathbb{E}\Big(\Phi\big(\sum_{k=1}^n 1^{\otimes (k-1)}\otimes f\otimes 1^{\infty}\big)\Big)\leq$$
$$\leq\sum_{n=1}^{\infty}\frac1{e\cdot n!}\cdot B^{n-1}n\mathbb{E}(\Phi(f))=e^{B-1}\mathbb{E}(\Phi(f)).$$
\end{proof}

\begin{lemma}\label{phimoment}
Suppose that $\Phi$ is an Orlicz function satisfying $\Delta_2$-condition. If $\{f_k\}_{k=1}^n$ is a sequence of independent random variables from
$L_{\Phi}(0,1)$ such that $\sum_{k=1}^n \mathbb P\{\text{supp} (f_k)\}\leq1$, then\footnote{Note that $\bigoplus_{k=1}^n |f_k|$ lives on the interval $(0,1).$} 
$$\mathbb{E}\Big(\Phi \big(\sum_{k=1}^n f_k\big)\Big) \leq C\mathbb E\Big(\Phi \big(\bigoplus_{k=1}^n f_k\big)\Big).$$
\end{lemma}
\begin{proof} Without loss of generality, $f_k\geq0,$ $1\leq k\leq n.$ Let $\{h_k\}_{k=1}^n$ be a sequence of independent random variables such that $h_k$ is equimeasurable with $K_{\rm class}(f_k)$ for every $1\leq k\leq n.$ By Lemma \ref{kmaj}, we have $f_k\prec h_k.$ It follows from Lemma \ref{pacific lemma} that
$$\sum_{k=1}^nf_k\prec\sum_{k=1}^n h_k.$$
Therefore, we have
$$\mathbb{E}\Big(\Phi \big(\sum_{k=1}^n f_k\big)\Big) \leq\mathbb{E}\Big(\Phi \big(\sum_{k=1}^n h_k\big)\Big).$$
It follows from Theorem \ref{DFclass} that
$$\mu(\sum_{k=1}^n h_k)=\mu(K_{\rm class}(\bigoplus_{k=1}^nf_k)).$$
Therefore, we have
$$\mathbb{E}\Big(\Phi \big(\sum_{k=1}^n f_k\big)\Big) \leq\mathbb{E}\Big(\Phi \big(K_{\rm class}f\big)\Big),\quad f=\bigoplus_{k=1}^nf_k.$$
The assertion follows now from Lemma \ref{phik}.
\end{proof}

The next lemma is somewhat similar to \cite[Proposition 14]{pacific}.

\begin{lemma}\label{norm}
Suppose that $\Phi$ is an Orlicz function satisfying $\Delta_2$-condition. If $\{f_k\}_{k=1}^n$ is a sequence of uniformly bounded independent random variables, then
$$\mathbb{E}\Big(\Phi\big(\sum_{k=1}^n f_k\big)\Big)\leq C_\Phi\Phi\Big(\|\bigoplus_{k=1}^n f_k\|_{L_1\cap L_{\infty}(0,\infty)}\Big).$$
\end{lemma}
\begin{proof} To lighten the notations, we write
$$\sup_{1\leq k\leq n}\|f_k\|_{\infty}=A,\quad\|f_k\|_1=\alpha_k,\quad\sum_{k=1}^n\|f_k\|_1=\alpha.$$

If $\alpha>A$, then $f_k\prec \alpha\chi_{[0,\alpha^{-1}\alpha_k]}$ for $1\leq k\leq n$.
Applying Lemma \ref{pacific lemma}, we obtain
$$\sum_{k=1}^nf_k\prec\alpha\sum_{k=1}^n1^{\otimes (k-1)}\otimes\chi_{[0,\alpha^{-1}\alpha_k]}\otimes 1^{\otimes (n-k)}.$$
Therefore,
$$\mathbb{E}\Big(\Phi\big(\sum_{k=1}^n f_k\big)\Big)\leq \mathbb{E}\Big(\Phi\big(\alpha\sum_{k=1}^n1^{\otimes (k-1)}\otimes\chi_{[0,\alpha^{-1}\alpha_k]}\otimes 1^{\otimes (n-k)}\big)\Big).$$
By Lemma \ref{phimoment}, we have that
$$\mathbb{E}\Big(\Phi\big(\sum_{k=1}^n f_k\big)\Big)\leq \mathbb{E}\Big(\Phi\big(\alpha\bigoplus_{k=1}^n\chi_{[0,\alpha^{-1}\alpha_k]}\big)\Big)=\Phi(\alpha).$$

If $\alpha\leq A,$ then $f_k\prec A\chi_{[0,A^{-1}\alpha_k]}$ for $1\leq k\leq n.$ It follows from Lemma \ref{pacific lemma} that
$$\sum_{k=1}^nf_k\prec A\sum_{k=1}^n1^{\otimes (k-1)}\otimes\chi_{[0,A^{-1}\alpha_k]}\otimes 1^{\otimes (n-k)}.$$
Therefore,
$$\mathbb{E}\Big(\Phi\big(\sum_{k=1}^n f_k\big)\Big)\leq \mathbb{E}\Big(\Phi\big(A\sum_{k=1}^n1^{\otimes (k-1)}\otimes\chi_{[0,A^{-1}\alpha_k]}\otimes 1^{\otimes (n-k)}\big)\Big).$$
By Lemma \ref{phimoment}, we have that
$$\mathbb{E}\Big(\Phi\big(\sum_{k=1}^n f_k\big)\Big)\leq \mathbb{E}\Big(\Phi\big(A\bigoplus_{k=1}^n\chi_{[0,A^{-1}\alpha_k]}\big)\Big)=\frac{\alpha}{A}\Phi(A).$$
Combining this estimates, we conclude the proof.
\end{proof}

Now, we are ready to prove our first main result, Theorem \ref{positive js rhs}.

\begin{proof}[Proof of theorem \ref{positive js rhs}] Without loss of generality, we may assume that
$$\mathbb{P}\{f_k=t\}=0,\quad t>0,\quad 1\leq k \leq n.$$
We set
$$f_{k,1}:=f_k\chi_{\{f_k>\mu(1,f)\}},\quad f_{k,2}:=f_k-f_{k,1},\quad 1\leq k\leq n.$$
The random variables $f_{k,1},$ $1\leq k\leq n,$ are positive and independent and so are the random variables $f_{k,2},$ $1\leq k\leq n.$ We also have that
$$\mu(\bigoplus_{k=1}^nf_{k,1})=\mu(f)\chi_{(0,1)},\quad\mu(\bigoplus_{k=1}^nf_{k,2})=\mu(\mu(f)\chi_{(1,\infty)}).$$
By \eqref{ineq13}
$$\Phi(u+v)\leq C_{\Phi} \big(\Phi(u)+\Phi(v)\big),\quad u,v>0.$$
Therefore, applying Lemma \ref{phimoment} and Lemma \ref{norm}, we obtain
\begin{eqnarray*}
\mathbb{E}\Big(\Phi\big(\sum_{k=1}^nf_k\big)\Big)
&\leq& C_{\Phi}\Big(\mathbb{E}\big(\Phi(\sum_{k=1}^nf_{k,1})\big)+\mathbb{E}\big(\Phi(\sum_{k=1}^nf_{k,2})\big)\Big)
\\&\leq& C_{\Phi}\Big(\mathbb{E}\big(\Phi(\bigoplus_{k=1}^nf_{k,1})\big)+\Phi\big(\|\bigoplus_{k=1}^n f_{k,2}\|_{L_1\cap L_{\infty}}\big)\Big)
\\&=& C_{\Phi}\Big[\mathbb{E}\Big(\Phi \big(\mu(f)\chi_{(0,1)} \big)\Big)+\Phi\Big(\|\mu(f)\chi_{(1,\infty)}\|_{L_1\cap L_{\infty}}\Big)\Big].
\end{eqnarray*}
The assertion follows now from the fact
$$\|\mu(f)\chi_{(1,\infty)}\|_{L_1\cap L_{\infty}}\leq\|f\|_1.$$
\end{proof}

We now turn to the converse inequality of Theorem \ref{positive js rhs}.

\begin{theorem}\label{positive js lhs}
Suppose that $\Phi$ is an Orlicz function satisfying $\Delta_2$-condition. If $\{f_k\}_{k=1}^n \subset L_{\Phi}[0,1]$, $n\in\mathbb{N},$ is a sequence of positive independent random variables, then
$$\mathbb{E}\Big(\Phi \big(\mu(f)\chi_{(0,1)} \big)\Big)+\Phi\Big(\|f\|_1\Big)
\leq 3\mathbb{E}\Big(\Phi\big(\sum_{k=1}^n {f_k}\big)\Big),\quad f=\bigoplus_{k=1}^n f_k.$$
\end{theorem}
\begin{proof} Let $f_{k,1},$ $f_{k,2},$ $1\leq k\leq n,$ be as in the proof of Theorem \ref{positive js rhs}. By the definition of $f_{k,1},$ we have
$$\sum_{k=1}^n\mathbb P\{\text{supp} (f_{k,1})\}\leq 1.$$
It follows from Lemma \ref{wbjohn lemma} that
$$\mathbb{E}\Big(\Phi \big(\mu(f)\chi_{(0,1)} \big)\Big)=\mathbb{E}\Big(\Phi \big(\bigoplus_{k=1}^nf_{k,1}\big)\Big)\leq 2\mathbb{E}\Big(\Phi \big(\sum_{k=1}^nf_{k,1}\big)\Big)\leq 2\mathbb{E}\Big(\Phi \big(\sum_{k=1}^nf_k\big)\Big).$$

On the other hand, we have
$$\Phi\Big(\|f\|_1\Big)=\Phi\left(\|\sum_{k=1}^nf_k\|_1\right)\leq\mathbb{E}\Big(\Phi\big(\sum_{k=1}^n {f_k}\big)\Big).$$
Here, the last inequality follows from the convexity of $\Phi.$ Combining these inequalities, we obtain the desired result.
\end{proof}

We now consider the case of symmetrically distributed random variables. The following theorem is our second main result.

\begin{theorem}\label{sym js rhs}
Suppose that $\Phi$ is an Orlicz function satisfying $\Delta_2$-condition. If $\{f_k\}_{k=1}^n \subset L_{\Phi}[0,1]$, $n\in\mathbb{N},$ is a sequence of symmetrically distributed independent random variables, then
\begin{equation}
\mathbb{E}\Big(\Phi\big(|\sum_{k=1}^nf_k|\big)\Big)\leq C_{\Phi}\Big[\mathbb{E}\Big(\Phi \big(\mu(f)\chi_{(0,1)} \big)\Big)+\Phi\Big(\|f\|_{L_1+L_2}\Big)\Big],\quad f=\bigoplus_{k=1}^nf_k.
\end{equation}
\end{theorem}

The key ingredient in the proof is the following lemma.

\begin{lemma}\label{like3.2}
Suppose that $\Phi$ is an Orlicz function satisfying $\Delta_2$-condition. If $\{f_k\}_{k=1}^n \subset L_{\Phi}[0,1]$, $n\in\mathbb{N},$ is a sequence of bounded symmetrically distributed independent random variables, then
$$\mathbb{E}\Big(\Phi\big(|\sum_{k=1}^n f_k\big)|\Big)\leq C_{\Phi}\Phi\Big(\|\bigoplus_{k=1}^n f_k\|_{L_2\cap L_{\infty}(0,\infty)}\Big).$$
\end{lemma}
\begin{proof} Let $\psi$ be a concave function such that $\psi'=\mu(K_{\rm class}1)$ and $\psi(0)=0.$ Consider the Marcinkiewicz space \cite{KPS}
$$M_{\psi}:=\{f\in S(\Omega, \mathbb{P}):\ \sup_{t>0}\frac1{\psi(t)}\int_0^t\mu(s,f)ds<\infty\},\quad \|f\|_{M_{\psi}}:=\sup_{t>0}\frac1{\psi(t)}\int_0^t\mu(s,f)ds.$$
Note that $f\prec\prec\|f\|_{M_{\psi}}\psi'.$ Recall that we identify $(\Omega, \mathbb{P})$ with $(0,1)$ equipped with Lebesgue measure. 

Define an operator $T:(L_2\cap L_{\infty})(0,\infty)\to M_{\psi}(0,1)$ by the following formula
$$Tf=\sum_{k=0}^{\infty}1^{\otimes 2k}\otimes (f\chi_{(k,k+1)}(\cdot-k))\otimes r\otimes 1^{\infty},$$
where $r:=\chi_{(0,1/2)}-\chi_{(1/2,1)}$. By \cite[Proposition 18]{pacific}, the operator $T:(L_2\cap L_{\infty})(0,\infty)\to M_{\psi}(0,1)$ is bounded. In particular, we have that
$$Tf\prec\prec C_{abs}\|f\|_{L_2\cap L_{\infty}}K_{\rm class}1.$$
Consequently,
$$\mathbb{E}(\Phi(Tf))\leq\mathbb{E}(\Phi(C_{abs}\|f\|_{L_2\cap L_{\infty}}K_{\rm class}1))\leq C_{\Phi}\Phi(C_{abs}\|f\|_{L_2\cap L_{\infty}})\leq C_{\Phi}\Phi(\|f\|_{L_2\cap L_{\infty}}).$$
Setting $f=\bigoplus_{k=1}^nf_k,$ we infer the assertion from the equality
$$T\big(\bigoplus_{k=1}^nf_k\big)=\sum_{k=1}^nf_k.$$
\end{proof}

We are now ready to prove Theorem \ref{sym js rhs}.

\begin{proof}[Proof of Theorem \ref{sym js rhs}] Without loss of generality, we may assume that
$$\mathbb{P}\{f_k=t\}=0,\quad t>0,\quad 1\leq k \leq n.$$
We set
$$f_{k,1}:=f_k\chi_{\{|f_k|>\mu(1,f)\}},\quad f_{k,2}:=f_k-f_{k,1},\quad 1\leq k\leq n.$$
The random variables $f_{k,1},$ $1\leq k\leq n,$ are independent and symmetrically distributed and so are the random variables $f_{k,2},$ $1\leq k\leq n.$ We also have that
$$\mu(\bigoplus_{k=1}^nf_{k,1})=\mu(f)\chi_{(0,1)},\quad\mu(\bigoplus_{k=1}^nf_{k,2})=\mu(\mu(f)\chi_{(1,\infty)}).$$
By the assumption, we have
$$\Phi(u+v)\leq C_{\Phi} \big(\Phi(u)+\Phi(v)\big),\quad u,v>0.$$
Therefore, applying Lemma \ref{phimoment} and Lemma \ref{like3.2}, we obtain
\begin{eqnarray*}
\mathbb{E}\Big(\Phi\big(|\sum_{k=1}^n {f_k}|\big)\Big)
&\leq& C_{\Phi}\Big(\mathbb{E}\big(\Phi(|\sum_{k=1}^n {f_{k,1}}|)\big)+\mathbb{E}\big(\Phi(|\sum_{k=1}^n {f_{k,2}}|)\big)\Big)
\\&\leq& C_{\Phi}\Big(\mathbb{E}\big(\Phi(|\bigoplus_{k=1}^n {f_{k,1}}|)\big)+\Phi\big(\|\bigoplus_{k=1}^n f_{k,2}\|_{L_2\cap L_{\infty}}\big)\Big)
\\&=& C_{\Phi}\Big[\mathbb{E}\Big(\Phi \big(\mu(f)\chi_{(0,1)} \big)\Big)+\Phi\Big(\|\mu(f)\chi_{(1,\infty)}\|_{L_2\cap L_{\infty}}\Big)\Big].
\end{eqnarray*}
The assertion follows now from the fact
$$\Phi\Big(\|\mu(f)\chi_{(1,\infty)}\|_{L_2\cap L_{\infty}}\Big)\leq\Phi\Big(\|\mu(f)\chi_{(1,\infty)}\|_2+\|\mu(f)\chi_{(1,\infty)}\|_{\infty}\Big)\leq$$
$$\leq \Phi\Big(\|\mu(f)\chi_{(1,\infty)}\|_2+\|\mu(f)\chi_{(0,1)}\|_1\Big)\leq\Phi(C_{abs}\|f\|_{L_1+L_2})\leq C_{\Phi}\Phi(\|f\|_{L_1+L_2}).$$
\end{proof}

The theorem below provides the opposite inequality to that of Theorem \ref{sym js rhs}.

\begin{theorem}\label{sym js lhs}
Suppose that $\Phi$ is an Orlicz function satisfying $\Delta_2$-condition. If $\{f_k\}_{k=1}^n \subset L_{\Phi}[0,1],$ $n\in\mathbb{N},$ is a sequence of symmetrically distributed independent random variables, then
$$\mathbb{E}\Big(\Phi \big(\mu(f)\chi_{(0,1)} \big)\Big)+\Phi\Big(\|f\|_{L_1+L_2}\Big)
\leq C \mathbb{E}\Big(\Phi\big(|\sum_{k=1}^n {f_k}|\big)\Big),\quad f=\bigoplus_{k=1}^nf_k.$$
\end{theorem}
\begin{proof} Let $f_{k,1},$ $1\leq k\leq n,$ be as in the proof of Theorem \ref{sym js rhs}. It follows from Lemma \ref{wbjohn lemma} that
$$\mathbb P\{|\bigoplus_{k=1}^n f_{k,1}|>t\}\leq 2\mathbb P\{\max_{1\leq k\leq n}|f_{k,1}|>t\}\leq 2\mathbb P\{\max_{1\leq k\leq n}|f_k|>t\},\quad t>0.$$
It follows now from  \cite[Lemma V.5.2]{an introduction} that
$$\mathbb P\{|\bigoplus_{k=1}^n f_{k,1}|>t\}\leq 4\mathbb P\{|\sum_{k=1}^nf_k|>t\}.$$
Therefore, we have
$$\mathbb{E}\Big(\Phi \big(\mu(f)\chi_{(0,1)} \big)\Big)=\mathbb{E}\Big(\Phi(\bigoplus_{k=1}^nf_{k,1})\Big)\leq 4\mathbb{E}\Big(\Phi\big(|\sum_{k=1}^n {f_k}|\big)\Big).$$

Applying  \cite[Theorem 1]{JS} to the space $L_1$ (see also \eqref{JS}), we infer that there exists an absolute constant $C>0$ such that
$$\|\bigoplus_{k=1}^nf_k\|_{L_1+L_2}\leq C\|\sum_{k=1}^nf_k\|_1.$$
Thus
\begin{eqnarray*}
\Phi(\|f\|_{L_1+L_2})\leq\Phi(C\|\sum_{k=1}^nf_k\|_1)\leq C_{\Phi}\mathbb E\Big(\Phi\big(|\sum_{k=1}^nf_k|\big)\Big).
\end{eqnarray*}
Combining these inequalities, we conclude the proof.
\end{proof}

Now we deal with the maximal inequalities and prove Theorem \ref{CM}.

\begin{proof}[Proof of Theorem \ref{CM}] Let $f_{k,1},f_{k,2}$ as the proof of Theorem \ref{positive js rhs}. We have
$$\Phi(x+y)\leq C_{\Phi}(\Phi(x)+\Phi(y)),\quad x,y>0$$
and, therefore,
$$\mathbb E\Phi\big(\sup_{1\leq k\leq n} {f_k}\big)\leq C_{\Phi}\Big(\mathbb E\Phi\big(\sup_{1\leq k\leq n} {f_{k,1}}\big)+\mathbb E\Phi\big(\sup_{1\leq k\leq n} {f_{k,2}}\big)\Big).$$
It follows from Lemma \ref{phimoment} that
$$\mathbb E\Phi\big(\sup_{1\leq k\leq n} f_{k,1}\big)\leq\mathbb E\Phi\big(\sum_{k=1}^nf_{k,1}\big)\leq C_{\Phi}\mathbb E\Phi\big(\bigoplus_{k=1}^nf_{k,1}\big)=C_{\Phi}\mathbb E\Phi\big(\mu(f)\chi_{(0,1)}\big).$$
Also, we have $\sup_{1\leq k\leq n}f_{k,2}\leq \mu(1,f)$ and, therefore,
$$\mathbb E\Phi\big(\sup_{1\leq k\leq n} {f_{k,2}}\big)\leq \Phi(\|\mu(f)\chi_{(0,1)}\|_1)\leq \mathbb E\Phi\big(\mu(f)\chi_{(0,1)}\big).$$
Therefore, we have
$$\mathbb E\Phi\big(\sup_{1\leq k\leq n} {f_k}\big)\leq \mathbb E\Phi\big(\mu(f)\chi_{(0,1)}\big).$$
To prove the converse inequality, note that it follows from Lemma \ref{wbjohn lemma} that
$$\mathbb E\Phi\big(\mu(f)\chi_{(0,1)}\big)=\mathbb E\Phi\big(\bigoplus_{k=1}^nf_{k,1}\big)\leq 2\mathbb E\Phi\big(\max_{1\leq k\leq n}f_{k,1}\big)\leq 2\mathbb E\Phi\big(\max_{1\leq k\leq n}f_k\big).$$
\end{proof}

As stated in introduction, the $\Delta_2$-condition is necessary in Theorem \ref{main results 1}. We now prove \eqref{sharpa} in Theorem \ref{sharp}.

\begin{proof}[Proof of Theorem \ref{sharp}(i)] Let $\{f_{k,n}\}_{k=1}^n$ be a sequence of independent random variables such that $f_{k,n}$ be equimeasurable with of $a\chi_{(0,1/n)}$ for some fixed $a>0$ and every $1\leq k\leq n$. Set for brevity $g:=K_{\rm class}1$. By  \cite[Lemma 6]{pacific} we have
$$\sum_{k=1}^nf_{k,n}\longrightarrow a g$$
in distribution. We also have $\bigoplus_{k=1}^nf_{k,n}=a.$ It follows from the Fatou theorem that 
$$\mathbb{E}\Phi\Big(ag\Big)\leq\liminf_{n\to\infty}\mathbb{E}\Phi\Big(\sum_{k=1}^nf_{k,n}\Big)\leq C_\Phi\liminf_{n\to\infty}\mathbb{E}\Phi\Big(\bigoplus_{k=1}^nf_{k,n}\Big)=C_{\Phi}\Phi(a).$$
Observe that $\int_0^1\Phi(ag)d{\mathbb P}\ge \Phi(2a)\cdot \mathbb{P}\{g\ge 2\}$. 
A combination of preceding inequalities yields that $\Phi$ satisfies $\Delta_2$-condition. The proof is complete.
\end{proof}

\section{Noncommutative (free) $\Phi$-moment inequalities}

In this section, we prove Theorem \ref{main results 2}, that is, $\Phi$-moment versions of Johnson-Schechtman inequalities for freely independent random variables. The symbols $L_p(\mathcal{M},\tau)$, $1\leq p \leq \infty$ and $L_\Phi(\mathcal{M},\tau)$ stand for noncommutative $L_p$-spaces and noncommutative Orlicz spaces respectively (see e.g. \cite{FK, SC, CS}).

\begin{lemma}\label{sssa} Let $n\ge 1$ and let $\{x_k\}_{k=1}^n$ and $\{y_k\}_{k=1}^n$ be sequences of freely independent positive random variables from $L_1(\mathcal{M},\tau).$
 If $\mu(y_k)\leq\mu(x_k)$ for $1\leq k\leq n,$ then
$$\mu\Big(\sum_{k=1}^ny_k\Big)\leq \mu\Big(\sum_{k=1}^nx_k\Big).$$
\end{lemma}
\begin{proof} Let $\mathcal{N}_k,$ $1\leq k\leq n,$ be finite von Neumann algebras. Let $\mathcal{N}=\star_{k=1}^n\mathcal{N}_k$ and let $i_k:L_{\infty}(0,1)\to\mathcal{N}_k$ be trace preserving $*-$homomorphisms. Let $u_k=i_k(\mu(x_k))$ and $v_k=i_k(\mu(y_k)).$ By Lemma \ref{FS}, we have
$$\mu(\sum_{k=1}^nx_k)=\mu(\sum_{k=1}^nu_k),\quad \mu(\sum_{k=1}^ny_k)=\mu(\sum_{k=1}^nv_k).$$
It is clear that $v_k=i_k(\mu(y_k))\leq i_k(\mu(x_k))=u_k,$ $1\leq k\leq n.$ Therefore, we have
$$0\leq\sum_{k=1}^nv_k\leq\sum_{k=1}^nu_k.$$
Thus,
$$\mu(\sum_{k=1}^nv_k)\leq\mu(\sum_{k=1}^nx_k).$$
This concludes the proof.
\end{proof}

If $x\in S^h(\M,\tau)$, then the projection onto the closure of the range of $\vert x\vert$ is called the support of $x$ and is denoted by ${\rm supp}(x)$.

\begin{lemma}\label{head lemma} Let $\{x_k\}_{k=1}^n\subset S(\mathcal{M},\tau)$ be a sequence of positive freely independent random variables. If $\sum_{k=1}^n\tau({\rm supp}(x_k))\leq 1,$ then\footnote{Here the symbol $\bigoplus_{k=1}^nx_k$ can be understood as a sum of any sequence $\{a_k\}_{k=1}^n\subset S(\mathcal{M},\tau)$ of positive operators whose supports are pairwise orthogonal and such that $\mu(a_k)=\mu(x_k)$, $1\leq k\leq n$.}
\begin{equation}\label{two-sided}
\frac1{10}\sigma_{\frac1{20}}\mu(\sum_{k=1}^nx_k)\leq\mu\Big(\bigoplus_{k=1}^nx_k\Big)\leq10\sigma_{20}\mu(\sum_{k=1}^nx_k).
\end{equation}
\end{lemma}
\begin{proof} Let $\{y_k\}_{k=1}^n\subset S(\mathcal{M},\tau)$ be a sequence of positive freely independent random variables such that $\mu(y_k)=\mu(K_{\rm free}x_k),$ $1\leq k\leq n.$ It follows from Lemma \ref{KB} that $\mu(y_k)\leq4\mu(x_k),$ $1\leq k\leq n.$ Therefore,
$$\mu(\sum_{k=1}^nx_k)\stackrel{L.\ref{sssa}}{\geq}\frac14\mu(\sum_{k=1}^ny_k)\stackrel{Th.\ref{DF}}{=}\frac14\mu\Big(K_{\rm free}\Big(\bigoplus_{k=1}^nx_k\Big)\Big)\stackrel{L.\ref{KB}}{\geq}\frac1{10}\sigma_{\frac1{20}}\mu\Big(\bigoplus_{k=1}^nx_k\Big).$$
Multiplying both parts of the inequality above by $10$ and applying to them $\sigma_{20}$, we arrive at the right hand side estimate in \eqref{two-sided}.

In order to prove the left hand side estimate in \eqref{two-sided}, let $\{z_k\}_{k=1}^n\subset S(\mathcal{M},\tau)$ be a sequence of freely independent random variables such that $\mu(z_k)=\sigma_{\frac1{20}}\mu(x_k),$ and let $u_k:=\bigoplus_{k=1}^{20}\mu(z_k)\in L_{\infty}(0,1), \ k=1,2,..., n$. Clearly, $\mu(u_k)=\mu(x_k)$, $1\leq k\leq n$ and therefore, without loss of generality, let $i_k:L_{\infty}(0,1)\to M_k$ be trace preserving $*$-homomorphisms from Lemma \ref{embedding} such that $i_k(u_k)=x_k$, $k=1,2,..., n$. Setting 
$$z_{k,l}=i_k(0^{\oplus(l-1)}\oplus \mu(z_k)\oplus 0^{\oplus (20-l)}),\quad 1\leq k\leq n,\ 1\leq l\leq 20
$$ we arrive at
$$\sum_{k=1}^nx_k=\sum_{l=1}^{20}\sum_{k=1}^nz_{k,l}$$ and, therefore,
$$\mu(\sum_{k=1}^nx_k)\leq \sum_{l=1}^{20}\sigma_{20}\mu(\sum_{k=1}^nz_{k,l}) = 20\sigma_{20}\mu(\sum_{k=1}^nz_k).$$
It follows from Lemma \ref{KB} that $\mu(z_k)\leq\frac52\mu(y_k),$ $1\leq k\leq n.$ Therefore, we have
$$\mu(\sum_{k=1}^nz_k)\stackrel{L.\ref{sssa}}{\leq}\frac52\mu(\sum_{k=1}^ny_k)\stackrel{Th.\ref{DF}}{=}\frac52\mu\Big(K_{\rm free}\Big(\bigoplus_{k=1}^nx_k\Big)\Big)\stackrel{L.\ref{KB}}{\leq}10\mu\Big(\bigoplus_{k=1}^nx_k\Big).$$
\end{proof}

\begin{proof}[Proof of Theorem \ref{main results 2} \eqref{mainfa}] Recall that $X=\bigoplus_{k=1}^nx_k.$ We can approximate each $x_k$ in the uniform norm with freely independent random variables without discrete spectrum. Thus, we may assume without loss of generality that $e_{\{t\}}(x_k)=0$ for every $t>0$ and for every $1\leq k\leq n.$ Equivalently, $e_{\{t\}}(X)=0$ for every $t>0.$ Set 
$$A_{1,k}=x_k e_{(\mu(1,X),\infty)}(x_k),\quad A_{2,k}=x_ke_{(0,\mu(1,X))}(x_k),\quad 1\leq k\leq n.$$
Random variables $A_{1,k}$ (respectively, $A_{2,k}$) $1\leq k\leq n,$ belong to the algebras generated by respective $A_k$ and are, therefore, freely independent. We have
\begin{equation}\label{disjoint_sums}
\mu\Big(\bigoplus_{k=1}^nA_{1,k}\Big)=\mu(X)\chi_{(0,1)},\quad \mu\Big(\bigoplus_{k=1}^nA_{2,k}\Big)=\mu(\mu(X)\chi_{(1,\infty)}).
\end{equation}
By \cite[Lemma 2.5(iv)]{FK}, we have 
$$\mu(\Phi(\sum_{k=1}^nA_{1,k}))=\Phi(\mu (\sum_{k=1}^nA_{1,k}))$$
and therefore, by the left hand side estimate in Lemma \ref{head lemma}, we obtain 
$$\mu(\Phi(\sum_{k=1}^nA_{1,k}))\leq \Phi (10\sigma_{20} \mu(\bigoplus_{k=1}^nA_{1,k})).$$
The latter estimate together with \eqref{disjoint_sums} yield
\begin{equation}\label{mainfa1}
\tau(\Phi(\sum_{k=1}^nA_{1,k}))\leq \mathbb{E}(\Phi(10\mu(X)\chi_{(0,1)}))\leq C_{\Phi}\mathbb{E}(\Phi(\mu(X)\chi_{(0,1)})),
\end{equation}
By \cite[Corollary 3.3]{SZ} we have
$$\|\sum_{k=1}^nA_{2,k}\|_{\infty}\leq 64\|\bigoplus_{k=1}^nA_{2,k}\|_{L_1\cap L_{\infty}}\leq 128\|X\|_1.$$
Therefore,
\begin{equation}\label{mainfa2}
\tau(\Phi(\sum_{k=1}^nA_{2,k}))\leq\Phi(\|\sum_{k=1}^nA_{2,k}\|_{\infty})\leq\Phi(128\|X\|_1)\leq C_{\Phi}\Phi(\|X\|_1).
\end{equation}
Since $\Phi$ satisfies $\Delta_2$-condition, it follows from \eqref{noncommineq13} that
$$\tau(\Phi(\sum_{k=1}^nx_k))\leq C_\Phi(\tau(\Phi(\sum_{k=1}^nA_{1,k}))+\tau(\Phi(\sum_{k=1}^nA_{2,k}))).$$
Combining the preceding estimate with \eqref{mainfa1} and \eqref{mainfa2} and recalling that $\sum_{k=1}^nx_k=\sum_{k=1}^n(A_{1,k}+A_{2,k})$, we arrive at
\begin{equation}\label{mainfa2-half}
\tau(\Phi(\sum_{k=1}^nx_k))\leq C_{\Phi}(\mathbb{E}(\Phi(\mu(X)\chi_{(0,1)}))+\Phi(\|X\|_1)).
\end{equation}
To complete the proof of Theorem \ref{main results 2} \eqref{mainfa}, it remains to prove the converse inequality to \eqref{mainfa2-half}. To that end, 
we observe that by Lemma \ref{sssa}, we have $$\tau(\Phi(\sum_{k=1}^nA_{1k}))\leq \tau(\Phi(\sum_{k=1}^nx_k)),$$ and by the right hand side estimate in Lemma \ref{head lemma}, we have $$\mathbb{E}(\Phi(\frac1{10}\sigma_{1/20}\mu(X)\chi_{(0,1)}))\leq \tau(\Phi(\sum_{k=1}^nA_{1k})).$$
Observing that 
$$
\mathbb{E}(\Phi(\mu(X)\chi_{(0,1)}))\leq C_{\Phi}\mathbb{E}(\Phi(\frac1{10}\sigma_{1/20}\mu(X)\chi_{(0,1)})).
$$
we arrive at
\begin{equation}\label{mainfa3}
\mathbb{E}(\Phi(\mu(X)\chi_{(0,1)}))\leq C_{\Phi}\tau(\Phi(\sum_{k=1}^nx_k)).
\end{equation}
Finally, by Jensen inequality \eqref{Jensen}, we have 
\begin{equation}\label{mainfa4}
\tau\Big(\Phi\big(\sum_{k=1}^nx_k\big)\Big)\geq\Phi\Big(\|\sum_{k=1}^nx_k\|_1\Big)=\Phi(\|X\|_1).
\end{equation}
Combining \eqref{mainfa3} and \eqref{mainfa4}, we arrive at the converse inequality to \eqref{mainfa2-half}.
\end{proof}

The following proposition will be needed for the proof of Theorem \ref{main results 2} \eqref{mainfb}.

\begin{proposition}\label{S} Let $(\mathcal{M},\tau)$ be a noncommutative probability space and let $\Phi$ be an Orlicz function satisfying $\Delta_2$-condition. If $(x_k)_{k=1}^n\subset L_{\Phi}(\mathcal{M},\tau)$ are freely independent symmetrically distributed random variables, then
$$\mathbb{E}(\Phi(\mu(X)\chi_{(0,1)}))\leq C_{\Phi}\tau(\Phi(\sum_{k=1}^nx_k)),\quad X=\bigoplus_{k=1}^nx_k.$$
\end{proposition}
\begin{proof} For every $t>0,$ the function $\Phi_t:=t^{-1}\Phi$ is an Orlicz function and, therefore, $L_{\Phi_t}$ is a (noncommutative) Orlicz space (a symmetric operator space equipped with a Fatou norm \cite{SC, DDP, CS}).  By \cite[Proposition 44]{SZ}, we have 
\begin{equation}\label{88}
\Big\|\mu(X)\chi_{(0,1)}\Big\|_{L_{\Phi_t}} \leq 3600 \Big\|\sum_{k=1}^nx_k\Big\|_{L_{\Phi_t}},\quad \forall t>0.
\end{equation}
Let $t>0$ be such that
\begin{equation*}
\tau(\Phi(3600\sum_{k=1}^nx_k))=t,
\end{equation*}
or, equivalently (see \cite[Chapter III, Theorem 3]{RR}), such that $\Big\|\sum_{k=1}^nx_k\Big\|_{L_{\Phi_t}}=\frac1{3600}$.
It follows from \eqref{88} that
$$\Big\|\mu(X)\chi_{(0,1)}\Big\|_{L_{\Phi_t}}\leq 1,$$
which is the same as (again referring to \cite[Chapter III, Theorem 3]{RR}) $$\mathbb{E}(\Phi(\mu(X)\chi_{(0,1)}))\leq t.$$
Hence, applying consequently \eqref{ineq13} we arrive at
$$\mathbb{E}(\Phi(\mu(X)\chi_{(0,1)}))\leq\tau(\Phi(3600\sum_{k=1}^nx_k))\leq C_{\Phi}\tau(\Phi(\sum_{k=1}^nx_k)).$$
\end{proof}

Now, we are in a position to furnish the proof of Theorem \ref{main results 2} \eqref{mainfb}. It has some similarities with the proof of Theorem \ref{main results 2} \eqref{mainfa}, however, some important details are different.

\begin{proof}[Proof of Theorem \ref{main results 2} \eqref{mainfb}] Without loss of generality, $e_{\{t\}}(|X|)=0$ for every $t>0.$ Set 
$$A_{1,k}=x_k e_{(\mu(1,X),\infty)}(|x_k|),\quad A_{2,k}=x_ke_{(0,\mu(1,X))}(|x_k|),\quad 1\leq k\leq n.$$
The sequence $\{A_{1,k}\}_{k=1}^n$ (respectively, $\{A_{2,k}\}_{k=1}^n$) consists of freely independent symmetrically distributed random variables and
$$\mu\Big(\bigoplus_{k=1}^nA_{1,k}\Big)=\mu(X)\chi_{(0,1)},\quad \mu\Big(\bigoplus_{k=1}^nA_{2,k}\Big)=\mu(\mu(X)\chi_{(1,\infty)}).$$
Using standard Jordan decomposition, we further write
$$
A_{1,k}=A^+_{1,k}-A^-_{1,k},\quad 1\leq k\leq n
$$
and observe that the sequence $\{A^+_{1,k}\}_{k=1}^n$ and  $\{A^-_{1,k}\}_{k=1}^n$ consist of freely independent positive random variables such that
$$
\mu\Big(\bigoplus_{k=1}^nA^+_{1,k}\Big),\ \mu\Big(\bigoplus_{k=1}^nA^-_{1,k}\Big)\leq \mu\Big(\bigoplus_{k=1}^nA_{1,k}\Big).
$$
Now, using the argument in the proof of Theorem \ref{main results 2} \eqref{mainfa} (see, in particular, \eqref{mainfa1} and preceding to it estimate) to justify the first inequality below, we obtain
\begin{equation}\label{mainfb1}
\tau(\Phi(\sum_{k=1}^nA^+_{1,k}))\leq C_{\Phi}\tau(\Phi \Big(\bigoplus_{k=1}^nA^+_{1,k}\Big))\leq C_{\Phi}\mathbb{E}(\Phi(\mu(X)\chi_{(0,1)})),
\end{equation}
and
\begin{equation}\label{mainfb2}
\tau(\Phi(\sum_{k=1}^nA^-_{1,k}))\leq C_{\Phi}\tau(\Big(\bigoplus_{k=1}^nA^-_{1,k}\Big))\leq C_{\Phi}\mathbb{E}(\Phi(\mu(X)\chi_{(0,1)}))
\end{equation}
Appealing now to \eqref{noncommineq13} and combining \eqref{mainfb1} and \eqref{mainfb2}, we arrive at
\begin{equation}\label{mainfb3}
\tau(\Phi(\sum_{k=1}^nA_{1,k}))\leq C_{\Phi}\mathbb{E}(\Phi(\mu(X)\chi_{(0,1)})).
\end{equation}
In order to deal with the sequence $\{A_{2,k}\}_{k=1}^n$, we firstly recall  that by \cite[Corollary 33(b)]{SZ} we have
$$\|\sum_{k=1}^nA_{2,k}\|_{\infty}\leq 64\|\bigoplus_{k=1}^nA_{2,k}\|_{L_2\cap L_{\infty}}\leq 128\|\mu(X)\chi_{(1,\infty)}\|_{L_2\cap L_{\infty}}
.$$
We  now appeal to a well known formula of T. Holmstedt \cite[Theorem 4.2]{H}, which in our special case yields
$$
\|\mu(X)\|_{L_1+L_2}\approx \|\mu(X)\chi_{[0,1]}\|_{L_1}+\|\mu(X)\chi_{[1,\infty]}\|_{L_2}.
$$
An immediate corollary of this formula is that $\|\mu(X)\chi_{(1,\infty)}\|_{L_2\cap L_{\infty}}\approx \|\mu(X)\|_{L_1+L_2}$ and
therefore,
\begin{equation}\label{mainfb4}
\tau(\Phi(\sum_{k=1}^nA_{2,k}))\leq\Phi(\|\sum_{k=1}^nA_{2,k}\|_{\infty})\leq C_{\Phi}\Phi(\|X\|_{L_1+L_2}).
\end{equation}
Hence, it follows from \eqref{mainfb3} and \eqref{mainfb4} that
\begin{equation}\label{mainfb5}
\tau(\Phi(\sum_{k=1}^nx_k))\leq C_{\Phi}(\mathbb{E}(\Phi(\mu(X)\chi_{(0,1)}))+\Phi(\|X\|_{L_1+L_2})).
\end{equation}
To complete the proof of Theorem \ref{main results 2} \eqref{mainfb}, we need to verify the converse inequality to \eqref{mainfb5}. To this end, recall that by Proposition \ref{S} we have
\begin{equation}\label{mainfb6}
\mathbb{E}(\Phi(\mu(X)\chi_{(0,1)}))\leq C_{\Phi}\tau(\Phi(\sum_{k=1}^nx_k)),
\end{equation}
and by \cite[Proposition 43]{SZ} we have
$$\|X\|_{L_1+L_2}\leq 64\|\sum_{k=1}^nx_k\|_1.$$
Combining the preceding inequality with \eqref{ineq13}, we infer 
$$
\Phi(\|X\|_{L_1+L_2}) \leq \Phi\Big(64\|\sum_{k=1}^nx_k\|_1\Big) \leq C_{\Phi}\Phi\Big(\|\sum_{k=1}^nx_k\|_1\Big).
$$
Now, recalling once more Jensen inequality \eqref{Jensen} and firstly estimating 
$$\Phi\Big(\|\sum_{k=1}^nx_k\|_1\Big)\leq \tau\Big(\Phi\big(\sum_{k=1}^nx_k\big)\Big),$$
and then combining preceding inequalities with \eqref{mainfb6}, we conclude that
$$ (\mathbb{E}(\Phi(\mu(X)\chi_{(0,1)}))+\Phi(\|X\|_{L_1+L_2})) \leq C_{\Phi}\tau(\Phi(\sum_{k=1}^nx_k))  .$$
\end{proof}

The $\Delta_2$-condition in Theorem \ref{main results 2} cannot be weakened. We now prove Theorem \ref{sharp} \eqref{sharpb}.
\begin{proof}[Proof of Theorem \ref{sharp} (ii)] Fix a real number $a>0.$ Let $\{x_{k,n}\}_{k=1}^n$  be a sequence of freely independent random variables such that $\mu (x_{k,n})=a\chi_{(0,1/n)}.$ The proof goes along the lines of that of Theorem \ref{sharp} (i). Instead of Fourier transform, we use free cumulants\footnote{Free cumulant $\kappa_m,$ $m\geq 1,$ is a polynomial expression in terms of the moments of the random variable $x$ (see (12.2) and Example 12.4 in \cite{NS1}). Its crucial feature (see Proposition 12.3 in \cite{NS1}) can be stated as follows: if the random variables $x_k,$ $1\leq k\leq n,$ are freely independent, then
$$\kappa_m(\sum_{k=1}^nx_k)=\sum_{k=1}^n\kappa_m(x_k).$$
In fact, the converse assertion also holds true (we do not need this fact). It follows from construction of the free cumulants in \cite{NS1} that the moments are also polynomial expressions in terms of the cumulants.}. By the definition of free cumulants (see the book \cite{NS1} or formula (3) in \cite{SZ}), we have that\footnote{In the latter formula, $NC(m)$ are so-called non-crossing partitions of the set $\{1,\cdots,m\}$ and ${\rm Moeb}(\pi,\mathbf{1}_m)$ are constant coefficients (we omit their definitions and refer the interested reader to the book \cite{NS1}). The symbol $|\cdot|$ in the formula below stands for the cardinality.}
$$\kappa_m(\sum_{k=1}^nx_{k,n})=n\kappa_m(a\chi_{(0,1/n)})=n\sum_{\pi\in NC(m)}{\rm Moeb}(\pi,\mathbf{1}_m)\prod_{V\in\pi}\tau((a\chi_{(0,\frac1n)})^{|V|})=$$
$$=na^m\sum_{\pi\in NC(m)}{\rm Moeb}(\pi,\mathbf{1}_m)n^{-|\pi|}\to a^m=\kappa_m(aK_{\rm free}1)$$
as $n\to\infty.$ Here, the equality $a^m=\kappa_m(aK_{\rm free}1)$ follows from \cite[Lemma 24]{SZ}. Using formula (5) in \cite{SZ}, we infer that
$$\tau((\sum_{k=1}^nx_{k,n})^m)\to\tau((aK_{\rm free}1)^m)$$
as $n\to\infty.$ Convergence of moments implies convergence in distribution (see Preliminaries). Hence,
$$\mu(\sum_{k=1}^nx_{k,n})\longrightarrow\mu(a K_{\rm free}1)$$
almost everywhere. We also have $\bigoplus_{k=1}^nx_{k,n}=a.$ It follows from the Fatou theorem, that 
$$\tau(\Phi(aK_{\rm free}1))=\tau(\Phi(\mu(a K_{\rm free}1)))\leq\liminf_{n\to\infty}\tau(\Phi(\mu(\sum_{k=1}^nx_{k,n})))=$$
$$=\liminf_{n\to\infty}\tau(\Phi(\sum_{k=1}^nx_{k,n}))\leq C_\Phi\liminf_{n\to\infty}\tau(\Phi(\bigoplus_{k=1}^nx_{k,n}))=C_{\Phi}\Phi(a),$$
where the last inequality is guaranteed by the assumption of Theorem \ref{sharp} (ii).
Using the same arguments as in the proof of Theorem \ref{sharp} (i), we firstly observe that
$$ \Phi(2a)\tau(e_{(2,\infty)}(K_{\rm free}1))\leq \tau(\Phi(aK_{\rm free}1)),$$
and then infer that $\Phi$ satisfies $\Delta_2$-condition. The proof is complete.
\end{proof}

\section{Johnson-Schechtman-type maximal inequalities}
\subsection{ $\Phi$-moment maximal inequalities}

\begin{proof}[Proof of Theorem \ref{Mright}] Let the random variables $A_{1,k},$ $1\leq k\leq n,$ be as in the proof of Theorem \ref{main results 2}. Take 
$$a=\sum_{k=1}^nA_{1,k}+\mu(1,X).$$
It is obvious that $a\geq x_k,$ $1\leq k\leq n.$ Since $\Phi$ satisfies $\Delta_2$-condition, it follows from Theorem \ref{main results 2} that
$$\tau(\Phi(a))\leq C_{\Phi}(\tau(\Phi(\sum_{k=1}^nA_{1,k}))+\Phi(\mu(1,X)))\leq C_{\Phi}(\tau(\Phi(\bigoplus_{k=1}^nA_{1,k}))+\Phi(\mu(1,X)))=$$
$$=C_{\Phi}(\mathbb{E}(\Phi(\mu(X)\chi_{(0,1)}))+\Phi(\mu(1,X)))\leq C_{\Phi}\mathbb{E}(\Phi(\mu(X)\chi_{(0,1)})).$$

In order to prove the converse inequality, we denote $p_k={\rm supp}(A_{1,k}).$ These are freely independent random variables. By  in \cite[Corollary 33]{SZ}, we have
$$\|\sum_{k=1}^np_k\|_{\infty}\leq 64\|\bigoplus_{k=1}^np_k\|_{L_1\cap L_{\infty}}\leq 64.$$
Obviously, $A_{1,k}\leq a$ and, therefore, $A_{1,k}\leq p_kap_k$ for $1\leq k\leq n.$ It follows from  \cite[Proposition 4.6(ii)]{FK} that
$$\tau(\Phi(A_{1,k}))\leq\tau(\Phi(p_kap_k))\leq\tau(p_k\Phi(a)p_k),\quad 1\leq k\leq n.$$
It follows from \eqref{disjoint_sums} that $\mathbb{E}(\Phi(\mu(X)\chi_{(0,1)}))=\sum_{k=1}^n\tau(\Phi(A_{1,k})$ and hence
$$\mathbb{E}(\Phi(\mu(X)\chi_{(0,1)}))\leq\sum_{k=1}^n\tau(p_k\Phi(a)p_k)=\tau(\Phi(a)\sum_{k=1}^np_k).$$
Therefore,
$$\mathbb{E}(\Phi(\mu(X)\chi_{(0,1)}))\leq\tau(\Phi(a))\cdot\|\sum_{k=1}^np_k\|_{\infty}\leq 64\tau(\Phi(a)).$$
\end{proof}

We also state a similar inequality for operator monotone functions.

\begin{proposition}\label{KPM}
Let $(x_k)_{k=1}^n$ be positive freely independent random variables. If $\Phi$ is an operator monotone function, then
$$\inf\big\{\tau(\Phi(a)):\ a\geq x_k,\ 1\leq k\leq n\big\}\approx_{C_{\Phi}} \mathbb{E}(\Phi(\mu(X)\chi_{(0,1)})),\quad X=\bigoplus_{k=1}^nx_k.$$
\end{proposition}
\begin{proof} That the left hand side does not exceed the right hand side can be proved as in Theorem \ref{Mright}. We only prove the converse inequality. Though visually similar to the proof of Theorem \ref{Mright}, this proof contains specific details worth emphasizing.

Let the random variables $A_{1,k},$ $1\leq k\leq n,$ be as in the proof of Theorem \ref{main results 2} and let $p_k={\rm supp}(A_{1,k}),$ $1\leq k\leq n.$ Since $\Phi$ is an operator monotone function, it follows that 
$$\Phi(A_{1,k})=p_k\Phi(A_{1,k})p_k\leq p_k\Phi(a)p_k,\quad 1\leq k\leq n.$$
Hence,
$$\mathbb{E}(\Phi(\mu(X)\chi_{(0,1)}))=\sum_{k=1}^n\tau(\Phi(A_{1,k}))\leq\sum_{k=1}^n\tau(p_k\Phi(a)p_k)=\tau(\Phi(a)\sum_{k=1}^np_k)$$
and, therefore,
$$\mathbb{E}(\Phi(\mu(X)\chi_{(0,1)}))\leq\tau(\Phi(a))\cdot\|\sum_{k=1}^np_k\|_{\infty}\leq 64\tau(\Phi(a)).$$
\end{proof}

Combining results of Theorem \ref{Mright} and Proposition \ref{KPM}, we emphasize the following important case.

\begin{corollary}\label{cp} Let $0<p\leq\infty$ and $(x_k)_{k=1}^n\subset L_p(\mathcal{M},\tau)$ be positive freely independent random variables. We have
$$\inf\big\{\|a\|_p:\ a\geq x_k,\ 1\leq k\leq n\big\}\approx_{C_p} \|\mu(X)\chi_{(0,1)}\|_p,\quad X=\bigoplus_{k=1}^nx_k.$$
\end{corollary}

\subsection{Maximal inequalities for quasi-Banach spaces}

We need the following condition on the quasi-Banach symmetric operator space $E.$ 

\begin{definition} Quasi-Banach symmetric operator space $E$ is called $p$-fully symmetric, $p>0,$ if, for every $x\in E$ and for every $y$ such that $y^p\prec\prec x^p$ we have $y\in E$ and also $\|y\|_E\leq\|x\|_E.$
\end{definition}

In the case when $1\leq p<\infty$ the notion above is well known and plays an important role in the classical interpolation criteria. For more information the reader is referred to \cite[Theorem 4.7]{DDP-Int}.

\begin{theorem} Let $E$ be a $p$-fully symmetric quasi-Banach operator space, $0<p<\infty$. If $(x_k)_{k=1}^n\subset E(\mathcal{M},\tau)$ is a sequence of positive freely independent random variables, then
$$\inf\big\{\|a\|_E:\ a\geq x_k,\ 1\leq k\leq n\big\}\approx_{C_E} \|\mu(X)\chi_{(0,1)}\|_E,\quad X=\bigoplus_{k=1}^nx_k.$$
\end{theorem}
\begin{proof} It is not hard to see that if $E$ is $p$-fully symmetric quasi-Banach operator space, then it is also $q$-fully symmetric quasi-Banach operator space for every $0<q<p$. Indeed, suppose that $z\in E$ and $g$ is such that $g^q\prec\prec z^q$. Then, well known results concerning submajorization yield that $(g^q)^{p/q}\prec\prec(z^q)^{p/q}$, which implies $g\in E$ and $\|g\|_E\leq \|z\|_E$. Therefore, if $E$ is $p$-fully symmetric quasi-Banach operator space for some $p\ge 1$, then $E$ is $p$-fully symmetric quasi-Banach operator space for every $0<p< 1$. Thus, it is sufficient to prove the assertion for $0<p<1.$ Let the random variables $A_{1,k},$ $1\leq k\leq n,$ be the same as in the proof of Theorem \ref{main results 2}, in particular, we have
$$
\mu(X)\chi_{(0,1)}=\mu((\bigoplus_{k=1}^nA_{1,k}).
$$ 
Setting
$$a:=\sum_{k=1}^nA_{1,k}+\mu(1,X).$$
we have $a\geq x_k,$ $1\leq k\leq n.$ By \cite[Proposition 28 ]{SZ}, there exists a constant $C_E$ such that
$$\|a\|_E\leq\mu(1,X)+\|\sum_{k=1}^nA_{1,k}\|_E\leq\mu(1,X)+C_E\|\bigoplus_{k=1}^nA_{1,k}\|_E=$$
$$=\mu(1,X)+C_E\|\mu(X)\chi_{(0,1)}\|_E\leq C_E\|\mu(X)\chi_{(0,1)}\|_E.$$

In order to prove the converse inequality, we denote $p_k={\rm supp}(A_{1,k}).$ These are freely independent random variables. By  in \cite[Corollary 33]{SZ} we have
$$\|\sum_{k=1}^np_k\|_{\infty}\leq 64\|\bigoplus_{k=1}^np_k\|_{L_1\cap L_{\infty}}\leq 64.$$
Since the function $t\rightarrow t^p$ is operator monotone, it follows that
\begin{equation}\label{c}
(\bigoplus_{k=1}^nA_{1,k})^p\leq\bigoplus_{k=1}^n p_ka^pp_k.
\end{equation}
Define an operator $T:(L_1+L_{\infty})(\mathcal{M},\tau)\to (L_1+L_{\infty})(\mathcal{M},\tau)$ by setting
$$Tx=\bigoplus_{k=1}^np_kxp_k.$$
For every positive $x\in L_1(\mathcal{M},\tau),$ we have
$$\|Tx\|_1=\|\bigoplus_{k=1}^np_kxp_k\|_1=\tau\Big(\bigoplus_{k=1}^np_kxp_k\Big)=\sum_{k=1}^n\tau(p_kxp_k)=\tau(x\sum_{k=1}^np_k)\leq$$
$$\leq\|x\|_1\|\sum_{k=1}^np_k\|_{\infty}\leq 64\|x\|_1.$$
Hence,
$$\|T\|_{L_1\to L_1}\leq 256.$$
For every $x\in L_{\infty}(\mathcal{M},\tau),$ we have
$$\|Tx\|_{\infty}=\|\bigoplus_{k=1}^np_kxp_k\|_{\infty}=\sup_{1\leq k\leq n}\|p_kxp_k\|_{\infty}\leq\|x\|_{\infty}.$$
Hence,
$$\|T\|_{L_{\infty}\to L_{\infty}}\leq 1.$$
Applying \cite[Proposition 4.1]{DDP}, we infer that $Tx\prec\prec 256x$ for every $x\in (L_1+L_{\infty})(\mathcal{M},\tau).$
Thus,
$$\mu^p(X)\chi_{(0,1)}=\mu((\bigoplus_{k=1}^nA_{1,k})^p)\leq\mu(T(a^p))\prec\prec 256a^p.$$
Since $E$ is $p$-fully symmetric, it follows that
$$\|\mu(X)\chi_{(0,1)}\|_E\leq 256^{1/p}\|a\|_E.$$
\end{proof}

\section{Johnson-Schechtman inequalities: symmetric quasi-Banach space case}\label{jsqb}

Our main motivation in this section is to extend the result of \cite[Theorem 37]{SZ} from symmetric to quasi-symmetric Banach spaces. The technique developed in \cite{SZ} to handle the sums of free independent random variables in Banach space setting  (see key intermediate results Theorem 14 and Lemma 36 in \cite{SZ}) fail to extend to the quasi-normed setting. It should be also emphasized that techniques developed in \cite{pacific} to handle the same problem for the classical (commutative) setting is also inapplicable in the free independent setting. Essentially distinct approach is required to estimate the sums of free random variables in the quasi-Banach spaces. In the present section, we develop such an approach.

The following lemma is well known and can be found e.g. in Section 4.C.1 in \cite{MOA}.

\begin{lemma}\label{finite dim majorization lemma} Let $0\leq x,y\in\mathbb{R}^n$ and $y\prec x,$ then
$$y=\sum_{\pi\in\mathfrak{S}_n}a(\pi)(x\circ \pi),$$
where $\mathfrak{S}_n$ is the set of all permutations of $\{1,\cdots,n\}$ and $a$ is a map from $\mathfrak{S}_n$ to $[0,1]$ with $\sum_{\pi\in\mathfrak{S}_n}a(\pi)=1.$
\end{lemma}

\begin{lemma}\label{sssb} Let $0\leq x_k,y_k\in L_1(\mathcal{M},\tau),$ $1\leq k\leq n,$ be freely independent random variables. If $y_k\prec x_k$ for $1\leq k\leq n,$ then
$$\sum_{k=1}^ny_k\prec\sum_{k=1}^nx_k.$$
\end{lemma}
\begin{proof} Using Lemma \ref{FS}, we can assume without loss of generality that $\mathcal{M}=\star_{k=1}^n\mathcal{M}_k$ and that there exists a trace preserving $*$-isomorphism $i_k:L_{\infty}(0,1)\to\mathcal{M}_k,$ such that
$$y_k=i_k\big(\mu(y_k)\big)\quad {\rm and}\quad x_k=i_k\big(\mu(x_k)\big),\quad 1\leq k\leq n.$$
Fix $N\in\mathbb{N}.$ For every $1\leq k\leq n,$ define the functions $u_{k,N}$ and $v_{k,N}$ by setting
$$u_{k,N}(t)=N\int_{\frac{i-1}{N}}^{\frac{i}{N}}\mu(s,x_k)ds,\quad v_{k,N}(t)=N\int_{\frac{i-1}{N}}^{\frac{i}{N}}\mu(s,y_k)ds,\quad\frac{i-1}{N}<t<\frac{i}{N},$$
for every $1\leq i\leq N.$ Define elements $x_{k,N},y_{k,N}\in\mathcal{M}_k,$ $1\leq k\leq N,$ by setting
$$x_{k,N}=i_k(u_{k,N}),\quad y_{k,N}=i_k(v_{k,N}).$$
It is easy to check that $v_{k,N}\prec u_{k,N},$ $1\leq k\leq n.$

It follows from Lemma \ref{finite dim majorization lemma} that there exists a mapping $a_k:\mathfrak{S}_N\to[0,1]$ such that 
$$v_{k,N}=\sum_{\pi_k\in\mathfrak{S}_N}a_k(\pi_k)(u_{k,N}\circ\pi_k),\quad \sum_{\pi_k\in\mathfrak{S}_N}a_k(\pi_k)=1.$$
Let us introduce notations
$$\pi=(\pi_1,\cdots,\pi_n)\in \mathfrak{S}_N^{\times n}=\underbrace{\mathfrak{S}_N\times\cdots\times\mathfrak{S}_N}_{\mbox{$n$ times}},$$
$$a(\pi)=a_1(\pi_1)\cdots a_n(\pi_n),\quad X_{k,N,\pi}=i_k(u_{k,N}\circ\pi_k).$$
In these notations, we have $\sum_{\pi\in\mathfrak{S}_N^{\times n}}a(\pi)=1$ and we can write
$$v_{k,N}=\sum_{\pi\in\mathfrak{S}_N^{\times n}}a(\pi)(u_{k,N}\circ\pi_k),\quad y_{k,N}=\sum_{\pi\in\mathfrak{S}_N^{\times n}}a(\pi)X_{k,N,\pi},\quad 1\leq k\leq n.$$
Therefore, 
$$\sum_{k=1}^ny_{k,N}=\sum_{k=1}^n\sum_{\pi\in\mathfrak{S}_N^{\times n}}a(\pi)X_{k,N,\pi}=\sum_{\pi\in\mathfrak{S}_N^{\times n}}a(\pi)\Big(\sum_{k=1}^nX_{k,N,\pi}\Big).$$
It follows from Theorem 3.3.3 in \cite{LSZ} that
$$\sum_{k=1}^ny_{k,N}\prec\sum_{\pi\in \mathfrak{S}_N^{\times n}}a(\pi)\mu\Big(\sum_{k=1}^nX_{k,N,\pi}\Big).$$
For a fixed $N,\pi,$ the random variables $X_{k,N,\pi},$ $1\leq k\leq n,$ are freely independent (observe that $X_{k,N,\pi}\in \M_k$). Since we also have $\mu(X_{k,N,\pi})=\mu(x_{k,N}),$ it follows from Lemma \ref{FS} that
$$\mu\Big(\sum_{k=1}^nX_{k,N,\pi}\Big)=\mu\Big(\sum_{k=1}^nx_{k,N}\Big).$$
Hence, we obtain
$$\sum_{k=1}^ny_{k,N}\prec\sum_{k=1}^nx_{k,N}.$$
As $N\to\infty,$ we have that $y_{k,N}\to y_k$ in $L_1(\mathcal{M},\tau).$ Therefore,
$$\sum_{k=1}^ny_{k,N}\to\sum_{k=1}^ny_k,\quad \sum_{k=1}^nx_{k,N}\to\sum_{k=1}^nx_k$$
in $L_1(\mathcal{M},\tau)$ as $N\to\infty.$ This concludes the proof.
\end{proof}

The following formula for the $p-$norm of a free Poisson random variable is certainly known. We provide a short proof for completeness.

\begin{lemma}\label{P} Let $\xi_u$ be a free Poisson random variable with parameter $u>1.$ We have
$$\|\xi_u\|_p\approx_{c_p} u\quad {\rm for }\quad  0<p<\infty.$$
\end{lemma}
\begin{proof} Without loss of generality, $u>4.$ Hence,
$$\frac14u\leq (1-\sqrt{u})^2\leq (1+\sqrt{u})^2\leq 4u.$$
It follows from \eqref{frpodef} that
$$\|\xi_u\|_p^p=\int_{\mathbb{R}}t^pdm_u(t)=\int_{(1-\sqrt{u})^2}^{(1+\sqrt{u})^2}\frac{t^p}{2\pi t}\sqrt{4u-(t-1-u)^2}\,dt\approx$$
$$\approx_{4^{p-1}}u^{p-1}\int_{(1-\sqrt{u})^2}^{(1+\sqrt{u})^2}\sqrt{4u-(t-1-u)^2}\,dt\stackrel{t-1-u=v}{=}u^{p-1}\int_{-2u^{1/2}}^{2u^{1/2}}\sqrt{4u-v^2}\,dv.$$
Since the right hand side equals $2\pi u^p,$ the assertion follows.
\end{proof}

\begin{lemma}\label{tail lemma} For every $p>0,$ if $(x_k)_{k=1}^n\subset L_p(\mathcal{M},\tau)$ is a sequence of positive freely independent random variables, then
$$\|\mu(X)\chi_{(1,\infty)}\|_1\leq c_p\|\sum_{k=1}^nx_k\|_p,\quad X=\bigoplus_{k=1}^n\mu(x_k).$$
Here, $c_p$ is the constant which depends only on $p.$
\end{lemma}
\begin{proof} We assume that $0<p<1$ (for $p\geq 1,$ the assertion is proved in \cite{SZ}). Without loss of generality, $\mu(1,X)>0.$ The random variables $y_k=\min\{x_k,\mu(1,X)\},$ $1\leq k\leq n,$ are also freely independent. Consider positive freely independent random variables $z_k,$ $1\leq k\leq n,$ such that
$$\mu(z_k)=\chi_{(0,\beta_k)},\quad {\rm and}\quad \beta_k=\frac{\|y_k\|_1}{\mu(1,X)},\quad 1\leq k\leq n.$$
Let $\xi_k,$ $1\leq k\leq n,$ be freely independent free Poisson random variables with parameters $\beta_k.$

Since $\|y_k\|_{\infty}\leq\mu(1,X)$ and since $\|y_k\|_1=\|\mu(1,X)z_k\|_1,$ it follows that $y_k\prec \mu(1,X)z_k,$ $1\leq k\leq n.$ It follows from Lemma \ref{sssb} and Lemma 25 in \cite{pacific} that
$$\sum_{k=1}^ny_k\prec\mu(1,X)\sum_{k=1}^n z_k\mbox{ and, therefore, }\|\sum_{k=1}^n y_k\|_p\geq\mu(1,X)\|\sum_{k=1}^n z_k\|_p.$$
It follows from Lemma \ref{KB} that $\mu(z_k)\leq \frac14\mu(\xi_k),$ $1\leq k\leq n.$ Hence,
$$\|\sum_{k=1}^n x_k\|_p\geq\|\sum_{k=1}^n y_k\|_p\geq\mu(1,X)\|\sum_{k=1}^n z_k\|_p\stackrel{L.\ref{sssa}}{\geq}\frac14\mu(1,X)\|\sum_{k=1}^n\xi_k\|_p.$$

Since the random variables $\xi_k,$ $1\leq k\leq n,$ are freely independent, it follows from Lemma \ref{frsum} that $\sum_{k=1}^n\xi_k$ is a free Poisson random variable with a parameter $u=\sum_{k=1}^n\beta_k.$ By construction, $u>1.$ Thus,
$$\|\sum_{k=1}^n x_k\|_p\geq \frac14\mu(1,X)\|\xi_u\|_p\stackrel{L.\ref{P}}{\geq}c_p\mu(1,X)\sum_{k=1}^n\beta_k\geq c_p\|\mu(X)\chi_{(1,\infty)}\|_1.$$
\end{proof}

The following theorem (a Johnson-Schechtman inequality for quasi-Banach symmetric operator spaces)  is our main result in this section. It extends \cite[Theorem 37]{SZ}.

\begin{theorem}
Let $E$ be an arbitrary symmetric quasi-Banach space and let $(x_k)_{k=1}^n\subset E(\mathcal{M},\tau)$ be a sequence of positive freely independent random variables. We have
$$\|\bigoplus_{k=1}^n\mu(x_k)\|_{Z_E^1}\leq C_E \big\|\sum_{k=1}^nx_k\big\|_E.$$
\end{theorem}
\begin{proof} Let
$$X=\bigoplus_{k=1}^n\mu(x_k).$$
Without loss of generality, $e_{\{t\}}(X)=0$ for every $t>0.$ Define random variables
$$A_{1k}=x_k e_{(\mu(1,X),\infty)}(x_k),\quad 1\leq k\leq n.$$ 
It follows from \eqref{head lemma} that
$$\mu(X)\chi_{[0,1]}=\mu\Big(\bigoplus_{k=1}^nA_{1k}\Big)\leq 10\sigma_{20}\mu(\sum_{k=1}^nA_{1k})\leq 10\sigma_{20}\mu(\sum_{k=1}^n x_k).$$
Therefore,
$$\|\mu(X)\chi_{(0,1)}\|_E\leq C_E\|\sum_{k=1}^nx_k\|_E.$$
Fix $0<p<1$ such that $E\subset L_p$ (see Lemma 24 in \cite{pacific}). It follows from Lemma \ref{tail lemma} that
$$\|\mu(X)\chi_{(1,\infty)}\|_1\leq c_p\|\sum_{k=1}^nx_k\|_p\leq C_E\|\sum_{k=1}^nx_k\|_E.$$
Combining these estimates, we conclude the proof.
\end{proof}

\section{Further remarks}

Let $\{f_k\}_{k=1}^n$ be a sequence of positive independent random variables satisfying the condition $\sum_{k=1}^n\mathbb{P}(\{f_k>0\})\leq 1.$ It is stated in Lemma \ref{wbjohn lemma} that 
\begin{equation}\label{M}
\mu\Big(\bigoplus_{k=1}^nf_k\Big)\leq \sigma_2\mu\Big(\max_{1\leq k\leq n}f_k\Big).
\end{equation}

At the moment, we are unaware of a free version of \eqref{M}. We state it as an open problem.

\begin{problem}\label{KP} Does there exist an absolute constant $C>1$ such that for every sequence $(x_k)_{k=1}^n\subset S(\mathcal{M},\tau)$ of positive freely independent random variables such that
$$\sum_{k=1}^n\tau\big(supp(x_k)\big)\leq 1,$$
we have
$$\mu\Big(\bigoplus_{k=1}^nx_k\Big)\leq C\sigma_C\inf\Big\{\mu(a): x_n\leq a, \forall n\geq1, a\in L^+_0(\M) \Big\}?$$
\end{problem}

\medskip

\noindent{\bf Acknowledgements}
This work was completed when the first named author was visiting UNSW. He would like to express his gratitude to the School of Mathematics and Statistics of UNSW for its warm hospitality.

\end{document}